\documentclass[intlim,righttag,10pt]{amsart}
\usepackage{amscd}
\usepackage{amssymb}
\usepackage[all]{xy}
\RequirePackage{mathrsfs}

\def\invlim{\mathop{\vtop{\ialign{##\crcr$\hfill{\lim}\hfil$\crcr
\noalign{\kern1pt\nointerlineskip}\leftarrowfill\crcr\noalign
{\kern -3pt}}}}\limits}
\def\dirlim{\mathop{\vtop{\ialign{##\crcr$\hfill{\lim}\hfil$\crcr
\noalign{\kern1pt\nointerlineskip}\rightarrowfill\crcr\noalign
{\kern -3pt}}}}\limits}
\def\lomapr#1{\smash{\mathop{\relbar\joinrel\longrightarrow}\limits^{#1}}}
 \def\verylomapr#1{\smash{\mathop{\relbar\joinrel\relbar\joinrel\relbar\joinrel\longrightarrow}\limits^{#1}}}

\def\epsilon{\varepsilon}
\let\mathcal\mathscr

\newtheorem{theorem}[equation]{Theorem}
 \newtheorem{lemma}[equation]{Lemma}
 \newtheorem{proposition}[equation]{Proposition}
 \newtheorem{corollary}[equation]{Corollary}

\theoremstyle{definition}
\newtheorem{definition}[equation]{Definition}
\theoremstyle{remark}
\newtheorem{remark}[equation]{Remark}

\newtheorem*{acknowledgments}{Acknowledgments}

\usepackage[T1]{fontenc}

\let\emptyset\varnothing

\let\cal\mathcal

\def\phi{\varphi}

\def\Q{{\bf Q}} \def\Z{{\bf Z}}
\def\zp{\Z_p} \def\qp{\Q_p}

\def\OC{{\mathcal{O}_C}}
\def\N{{\bf N}}
\def\O{{\cal O}}

\def\epsilon{\varepsilon}

\def\Ainf{{\mathbb{A}_{\rm inf}}}

 \def\A{{\bf A}}

\def\rg{{\rm R}\Gamma}

\def\verylomapr#1{\smash{\mathop{\relbar\joinrel\relbar\joinrel\relbar\joinrel\longrightarrow}\limits^{#1}}}

\newcommand{\ao}{{A\Omega_{\calx}}}

\newcommand{\calx}{\mathfrak{X}}

\newcommand{\R}{\mathrm {R} }

\newcommand{\LL}{\mathrm {L} }

\newcommand{\ovk}{\overline{K} }

\newcommand{\gp}{\operatorname{gp} }

 \newcommand{\cl}{\operatorname{cl} }

 \newcommand{\holim}{\operatorname{holim} }

  \newcommand{\proet}{\operatorname{pro\acute{e}t}  }

 \newcommand{\eet}{\operatorname{\acute{e}t} }

 \newcommand{\dlog}{\operatorname{dlog} }

 \newcommand{\Spf}{\operatorname{Spf} }
 
 \newcommand{\Hom}{\operatorname{Hom} }

 \newcommand{\can}{ \operatorname{can} }

 \newcommand{\id}{ \operatorname{Id} }

\newcommand{\hk}{\operatorname{HK} }
\newcommand{\dr}{\operatorname{dR} }

  \newcommand{\htt}{\operatorname{HT}}

 \newcommand{\sh}{{\mathcal{H}}}
 \newcommand{\sg}{{\mathcal{G}}}

  \newcommand{\su}{{\mathcal{U}}}

 \newcommand{\so}{{\mathcal O}}

 \newcommand{\wt}{\widetilde}
 \newcommand{\wh}{\widehat}

      \def\A{{\bf A}}

\def\invlim{\mathop{\vtop{\ialign{##\crcr$\hfill{\lim}\hfil$\crcr
\noalign{\kern1pt\nointerlineskip}\leftarrowfill\crcr\noalign
{\kern -3pt}}}}\limits}
\def\dirlim{\mathop{\vtop{\ialign{##\crcr$\hfill{\lim}\hfil$\crcr
\noalign{\kern1pt\nointerlineskip}\rightarrowfill\crcr\noalign
{\kern -3pt}}}}\limits}
\def\lomapr#1{\smash{\mathop{\relbar\joinrel\longrightarrow}\limits^{#1}}}
 \def\verylomapr#1{\smash{\mathop{\relbar\joinrel\relbar\joinrel\relbar\joinrel\longrightarrow}\limits^{#1}}}

\def\epsilon{\varepsilon}
\let\mathcal\mathscr

\numberwithin{equation}{section}

\begin{document}
\title[Integral $p$-adic \'etale cohomology of Drinfeld symmetric spaces]
{Integral $p$-adic \'etale cohomology of Drinfeld symmetric spaces}
\author{Pierre Colmez}
\address{CNRS, IMJ-PRG, Sorbonne Universit\'e, 4 place Jussieu,
75005 Paris, France}
\email{pierre.colmez@imj-prg.fr}
\author{Gabriel Dospinescu}
\address{CNRS, UMPA, \'Ecole Normale Sup\'erieure de Lyon, 46 all\'ee d'Italie, 69007 Lyon, France}
\email{gabriel.dospinescu@ens-lyon.fr}
\author{Wies{\l}awa Nizio{\l}}
\address{CNRS, UMPA, \'Ecole Normale Sup\'erieure de Lyon, 46 all\'ee d'Italie, 69007 Lyon, France}
\email{wieslawa.niziol@ens-lyon.fr}
 \thanks{This research was partially supported by the  project  ANR-14-CE25 and the NSF grant No. DMS-1440140.}
\begin{abstract}
We compute the  integral $p$-adic \'etale cohomology of  Drinfeld symmetric spaces of any dimension. This refines the computation of  the rational $p$-adic \'etale cohomology from \cite{CDN3}. The main tools are: the computation of the integral de Rham cohomology from \cite{CDN3} and the integral $p$-adic comparison theorems of Bhatt-Morrow-Scholze and \v{C}esnavi\v{c}ius-Koshikawa which replace the quasi-integral comparison theorem of Tsuji used in \cite{CDN3}. 
\end{abstract}
\setcounter{tocdepth}{2}

\maketitle

{\Small
\tableofcontents
}

\section{Introduction}

  Let $p$ be a prime number,
    $K$ a finite extension of 
   $\qp$,  and $C$ the $p$-adic completion of an algebraic closure $\ovk$  of $K$.
   Drinfeld's symmetric space of dimension $d$ over $K$ is the rigid analytic variety 
    $${\mathbb H}^d_K:={\mathbb P}^d_K\setminus \cup_{H\in \mathcal{H}} H,$$
    where 
    $\mathcal{H}$ is the space of $K$-rational hyperplanes
         in $K^{d+1}$. It is equipped with an action of 
                   $G={\rm GL}_{d+1}(K)$.
                  One of the main results of \cite{CDN3} is the description of the $G\times \mathcal{G}_K$-modules
                  $H^i_{\eet}({\mathbb H}^d_C, \qp(i))$, where ${\mathbb H}^d_C:={\mathbb H}^d_{K,C}$ and $\mathcal{G}_K={\rm Gal}(\overline{K}/K)$. The analogous result for  $\ell$-adic \'etale cohomology, $\ell\neq p$,  is a classical result of Schneider and Stuhler \cite{SS}. It relies on the fact that  $\ell$-adic \'etale cohomology satisfies a homotopy property with respect to the open ball (a fact that is false for $p$-adic \'etale cohomology).
                  
                  The goal of this paper is to refine our   result, by describing the integral $p$-adic \'etale cohomology groups $H^i_{\eet}({\mathbb H}^d_C, \zp(i))$.
                  Recall that, for  $i\geq 0$, there is a natural generalized Steinberg representation ${\rm Sp}_i(\zp)$ of $G$ (see Section \ref{Steinberg} for the precise definition). We endow it with the trivial action of $\mathcal{G}_K$ and we write 
                  ${\rm Sp}_i(\zp)^*$ for its $\zp$-dual. 
                  
                  The main result of this paper is the following: 
                                    
\begin{theorem}\label{key11}
Let  $i\geq 0$. There are compatible  topological isomorphisms of $G\times \mathcal{G}_K$-modules
$$H^i_{\eet}({\mathbb H}^d_C, \zp(i))\simeq {\rm Sp}_i(\zp)^*, \quad H^i_{\eet}({\mathbb H}^d_C, {\mathbf F}_p(i))\simeq {\rm Sp}_i(\mathbf{F}_p)^*,  $$
compatible with the isomorphism $H^i_{\eet}({\mathbb H}^d_C, \qp(i))\simeq {\rm Sp}_i(\zp)^*\otimes_{\zp} \qp$ 
from \cite{CDN3}.  In particular, for $i>d$ these cohomology groups are trivial.
\end{theorem}
If $d=1$ and $i=1$, this result is due to Drinfeld~\cite{Drinfeld} 
(with a shaky proof corrected in~\cite{FDP}; see also~\cite[Th.\,1.7]{CDN1}).
                  
\medskip
                                     
\noindent{\em \'Etale cohomology and $A_{\rm inf}$-cohomology}.
We will describe now the key ideas and difficulties occurring in the proof of Theorem \ref{key11}.
 As in \cite{CDN3}, a key input is the 
pro-ordinarity of the standard semistable formal model $\mathfrak{X}_{\mathcal{O}_K}$ of ${\mathbb H}^d_K$, a result 
due to Grosse-Kl\"onne \cite{GK1}.
 More precisely, he proved that 
\begin{equation}
\label{GK00}
H^i(\mathfrak{X}_{\so_K}, \Omega^j_{\mathfrak{X}_{\mathcal{O}_K}})=0,\quad i\geq 1, j \geq 0,
\end{equation}
where $\Omega^{\bullet}_{\mathfrak{X}_{\mathcal{O}_K}}$ is the logarithmic de Rham complex of 
$\mathfrak{X}_{\mathcal{O}_K}$ over $\mathcal{O}_K$ (for the canonical log-structures of $\mathfrak{X}_{\so_K}$ and $\so_K$).
 One easily infers from this that 
$\mathfrak{X}_{\so_K}$ is ordinary in the usual sense \cite{CDN3}.
 The strongest (and easiest) integral $p$-adic comparison theorems are available for ordinary varieties, making it natural to try to adapt them to $\mathfrak{X}_{\mathcal{O}_K}$.
 Nevertheless, the fact that $\mathfrak{X}_{\mathcal{O}_K}$ is not quasi-compact seems to be a serious obstacle in implementing the usual strategy \cite[Ch. 7]{BK} to our setup.
 The syntomic method, suitably adapted \cite{CDN3}, works well only up to some absolute constants, and reduces the computation of $H^i_{\eet}({\mathbb H}^d_C, \qp(i))$ to that of the (integral) Hyodo-Kato cohomology of the special fiber of $\mathfrak{X}_{\mathcal{O}_K}$, which was done  in \cite{CDN3}.
 The latter computation can be done integrally and also shows that the de Rham cohomology of $\mathfrak{X}_{\mathcal{O}_K}$ is $p$-torsion-free. 
           
The results of Bhatt-Morrow-Scholze \cite{BMS1} (adapted to the semistable reduction setting by \v{C}esnavi\v{c}ius-Koshikawa \cite{CK}) show that, 
for proper rigid analytic varieties with semistable reduction, if  the de Rham cohomology of the semistable  integral model is $p$-torsion free (equivalently, if the integral Hyodo-Kato  cohomology of the special fiber is $p$-torsion free) so  is the $p$-adic \'etale cohomology of the generic fiber.
 Combined with \cite{CDN3} and with the rigidity of $G$-invariant lattices in ${\rm Sp}_i(\qp)$ (a result due to Grosse-Kl\"onne \cite{GKir}), this would yield our main result.
 The problem is that the proofs in \cite{BMS1} and \cite{CK} rely  on the properness of the varieties and it is not clear how to adapt them to our context.
 However, the key actor in {\em loc. cit.}  makes perfect sense: the $A_{\rm inf}$-cohomology.
 One then needs a way to read the $p$-adic \'etale cohomology in terms of the 
$A_{\rm inf}$-cohomology, which can be done even for non quasi-compact varieties thanks to a remarkable (especially due to its simplicity!) formula in \cite{BMS2} (the way $p$-adic \'etale cohomology and $A_{\rm inf}$-cohomology are related in \cite{BMS1} is rather different and does not seem to be very useful in our case).
 This reduces the proof of our main theorem to the computation of the $A_{\rm inf}$-cohomology. 
           
More precisely, let $A_{\rm inf}=W(\mathcal{O}_C^{\flat})$ be Fontaine's ring associated to $C$.
The choice of a compatible system of 
primitive $p$-power roots of unity $(\zeta_{p^n})_{n}$ gives rise to an element $\mu=[\varepsilon]-1\in A_{\rm inf}$ (where $\varepsilon$ corresponds to $(\zeta_{p^n})_{n}$ under the identification $\mathcal{O}_C^{\flat}=\varprojlim_{x\mapsto x^p} \mathcal{O}_C$).
This, in turn, induces a modified Tate twist $M\to M\{i\}:=M\otimes_{A_{\rm inf}} A_{\rm inf}\{i\}$, $i\geq 0$,  on the category of $A_{\rm inf}$-modules, where 
$A_{\rm inf}\{1\}:= \frac{1}{\mu}A_{\rm inf}(1)$, $A_{\rm inf}\{i\}:=A_{\rm inf}\{1\}^{\otimes i}$.
 Let $X={\mathbb H}^d_C$ and $\mathfrak{X}=\mathfrak{X}_{\mathcal{O}_K}\otimes_{\mathcal{O}_K} \mathcal{O}_C$.
 Using the projection from the pro-\'etale site of $X$ to the \'etale site of $\mathfrak{X}$ and a relative version of Fontaine's construction of the ring $A_{\rm inf}$, one constructs in \cite{BMS1}, \cite{CK} a complex of sheaves of $A_{\rm inf}$-modules
           $\ao$ on the \'etale site of $\mathfrak{X}$, which allows one to interpolate between \'etale, crystalline, and de Rham cohomology of $X$ and $\mathfrak{X}$. 
           
The technical result we prove is then:
           
\begin{theorem}\label{key2}
Let $i\geq 0$. 
There is a topological $\phi^{-1}$-equivariant isomorphism of $G\times \mathcal{G}_K$-modules
$$H^i_{\eet}(\mathfrak{X}, \ao\{i\})\simeq A_{\rm inf}\widehat{\otimes}_{\zp} {\rm Sp}_i(\zp)^*.$$
\end{theorem}
Theorem \ref{key11} is now obtained from this and the description of $p$-adic nearby cycles  in \cite{BMS2} in terms of $\ao$ (a twisted version of the Artin-Schreier exact sequence): an exact sequence
\begin{equation} \label{evening1}
0\to {H^{i-1}_{\eet}(\calx, \ao\{i\})}/(1-\phi^{-1})\to H^i_{\eet}(X, \zp(i))\to H^i_{\eet}(\calx, \ao\{i\})^{\phi^{-1}=1}\to 0.
\end{equation}

\medskip                     
\noindent
{\em Proof of Theorem \ref{key2}.} We end the introduction by briefly explaining the key steps in the proof of 
Theorem~\ref{key2}. Fix $i\geq 0$ and write for simplicity 
$M=H^i(\calx, \ao\{i\})$.  This is an $A_{\rm inf}$-module, which is derived $\tilde{\xi}$-complete, for $\tilde{\xi}=\phi(\mu)/\mu$. 
                                    
In the first step, we interpret (following Schneider-Stuhler \cite{SS} and Iovita-Spiess \cite{IS}) ${\rm Sp}_i(\zp)^*$ as a suitable quotient of the space of $\zp$-valued measures on $\mathcal{H}^{i+1}$ (recall that $\mathcal{H}$ is the space of $K$-rational hyperplanes
in $K^{d+1}$). This allows us to construct an  \'etale  regulator (an "integration of \'etale symbols") map 
$$ r_{\eet}: {\rm Sp}_i(\zp)^*\to H^i_{\eet}(X, \Z_p(i)) $$
which induces a regulator map 
\begin{equation}
\label{intro-2}r_{\inf}: A_{\rm inf}\widehat{\otimes}_{\zp} {\rm Sp}_i(\zp)^*\to H^i_{\eet}(\mathfrak{X}, \ao\{i\}).
\end{equation}
        
  To prove that $r_{\inf}$ is an isomorphism we use the derived Nakayama Lemma:  since both sides of (\ref{intro-2}) are derived  $\tilde{\xi}$-complete it suffices to show that $r_{\inf}$  is a quasi-isomorphism when reduced modulo $\tilde{\xi}$ (in the derived sense). That is, that the morphism
$$
r_{\inf}\otimes^{\rm L}\id_{A_{\rm inf}/\tilde{\xi}}: (A_{\rm inf}\widehat{\otimes}_{\zp} {\rm Sp}_i(\zp)^*)\otimes_{A_{\rm inf}}^{\rm L}(A_{\rm inf}/\tilde{\xi})\to H^i_{\eet}(\mathfrak{X}, \ao\{i\})\otimes_{A_{\rm inf}}^{\rm L}(A_{\rm inf}/\tilde{\xi})
$$
is a quasi-isomorphism. 
         To compute  the naive reduction $\overline{r}_{\inf}$ modulo $\tilde{\xi}$ of (\ref{intro-2}) we use the Hodge-Tate specialization of 
                  $\ao$, which  identifies  $H^i(\ao/\tilde{\xi})$ with the (twisted) sheaf of $i$'th logarithmic  differential forms on $\mathfrak{X}$. And, globally, those  are well controlled by the acyclicity result  (\ref{GK00}).
     Combined with  a compatibility between the \'etale and the Hodge-Tate Chern class maps  and the Hodge-Tate specialization this implies  that    $\overline{r}_{\inf}$ is isomorphic to 
          the Hodge-Tate regulator
         $$
  r_{\htt}: \so_C\widehat{\otimes}_{\zp} {\rm Sp}_i(\zp)^*\to H^0_{\eet}(\calx, \Omega^i_{\calx}).
         $$
         And this we have shown to be an isomorphism in \cite{CDN3}. 
         
         Along the way we also compute that the target $H^i_{\eet}(\mathfrak{X}, \ao\{i\})$  of $r_{\rm inf}$ is  $\tilde{\xi}$-torsion free. Since the domain $A_{\rm inf}\widehat{\otimes}_{\zp} {\rm Sp}_i(\zp)^*$ of $r_{\inf}$ is also $\tilde{\xi}$-torsion free this shows that $r_{\inf}\otimes^{\rm L}\id_{A_{\rm inf}/\tilde{\xi}}\simeq \overline{r}_{\rm inf}$ and hence, by the above,  it is a quasi-isomorphism, as wanted.
                        
\begin{acknowledgments}W.N. would like to thank MSRI, Berkeley, for hospitality during Spring 2019 semester when parts of this paper were written. We would like to thank 
Bhargav Bhatt for suggesting that derived completions could simplify our original proof (which they did !). We thank K\k{e}stutis \v{C}esnavi\v{c}ius and Matthew Morrow for helpful discussions related to the subject of this paper.
 \end{acknowledgments}

\noindent
{\em Notation and conventions.} Throughout the paper $p$ is a fixed prime. $K$ is a   finite extension of $\qp$ with the ring of integers $\so_K$; $C$ is the $p$-adic completion of  an algebraic closure $\ovk$ of $K$. 

 All formal schemes are $p$-adic.    A   formal scheme  over $\so_K$ is called {\em semistable} if,  locally for the Zariski topology,  it admits \'etale maps to the formal spectrum 
    $\Spf(\so_K\{X_1,\ldots,X_n\}/(X_1\cdots X_r-\varpi))$, $1\leq r\leq n$, where $\varpi$ is a uniformizer of $K$. We equip it with the log-structure coming from the special fiber.

  If $A$ is a ring and $f\in A$ is a regular element (i.e.,  nonzero divisor) and $T\in D(A)$, we will often write $T/f$ for $T\otimes^{\rm L}_AA/f$ if there is no confusion.

    \section{Preliminaries} 

\subsection{Derived completions and the d\'ecalage functor}
\subsubsection{Derived completions.}\label{derived-complete}
We will need the following derived version of completion\footnote{The terminology here  is misleading. The derived $I$-completion is not given by $M\mapsto \holim_n(M\otimes^{\rm L}_AA/I^n)$, as one would naturally guess.}:
\begin{definition}{\rm (\cite[091S]{SP})}
Let $I$ be an ideal of a ring $A$. We say that $M\in D(A)$ is {\em derived $I$-complete} if for all $f\in I$ $$
\holim(\cdots\to M\stackrel{f}{\to} M\stackrel{f}{\to}M\stackrel{f}{\to} M)=0.
$$
\end{definition}

Let $A$ be a ring and let $I\subset A$ be an ideal. We list the following basic properties of derived completions \cite[091N]{SP}:
\begin{enumerate}
\item Let $M$ be an $A$-module.  If $M$ is classically $I$-complete, i.e., the map $M\to \invlim_nM/I^n$ is an isomorphism,  then $M$ is also derived $I$-complete; the converse is true if $M$ is $I$-adically separated.
\item The collection of all derived $I$-complete $A$-complexes forms a full triangulated subcategory of $D(A)$.
\item  $M\in D(A)$ is derived $I$-complete if and only if so are its cohomology groups $H^i(M)$, $i\in \Z$.
\item ({\em Derived Nakayama Lemma}) A derived $I$-complete complex $M\in D(A)$ is $0$ if and only if $M\otimes^{\rm L}_{A}A/I\simeq 0$.
\item If $I$ is generated by $x_1,...,x_n\in A$, then $M\in D(A)$ is derived $I$-complete if and only if 
$M$ is derived $(x_i)$-complete for $1\leq i\leq n$. 
\item If $f$ is a morphism of ringed topoi, the functor $\R f_*$ commutes with derived completions \cite[0944]{SP}.
\end{enumerate}

   \subsubsection{The Berthelot-Deligne-Ogus d\'ecalage functor.}
    
    For any ring 
    $A$ and any regular element $f\in A$ there is a functor $\LL\eta_f: D(A)\to D(A)$ (which is \textbf{not} exact) with the key property \cite[Lemma 6.4]{BMS1} that there is a functorial isomorphism\footnote{Depending on $f$, not only on the ideal $fA$. If we want to avoid this, the "correct" isomorphism is 
    $$H^i(\LL\eta_f(T))\simeq (H^i(T)/H^i(T)[f])\otimes_{A} (f^i),$$
    where $(f^i)\subset A[1/f]$ is the fractional $A$-ideal generated by $f^i$.}
    $$H^i(\LL\eta_f(T))\simeq H^i(T)/(H^i(T)[f]),$$
    where $M[f]:=\{x\in M|\, fx=0\}$. Concretely, choose a representative 
 $T^{\bullet}$ of $T\in D(A)$ such that $T^i[f]=0$ for all $i$, and consider the sub-complex
 $\eta_f(T^{\bullet})\subset T^{\bullet}[1/f]$ defined by 
    $$\eta_f(T^{\bullet})^i=\{x\in f^iT^i| \, dx\in f^{i+1}T^{i+1}\}.$$
    Its image $\LL\eta_{f}(T)$ in $D(A)$ depends only on $T$. 
      
      We list the following properties of the above construction (sometimes extended naturally to ringed topoi):

 \begin{enumerate}
 \item $\LL\eta_f$ commutes with truncations and with restriction of scalars\footnote{The latter means that $\alpha_*(\LL\eta_{\alpha(f)}(M))\simeq \LL\eta_f(\alpha_*M)$ for $M\in D(B)$ and $\alpha: A\to B$ a map of rings such that $\alpha(f)\in B$ is regular.}. Moreover, 
 $\LL\eta_f(\LL\eta_g(T))\simeq \LL\eta_{fg}(T)$ for $f,g\in A$ regular elements and $T\in D(A)$. 
   \item For all $T\in D(A)$, we have $\LL\eta_f(T)[1/f]\simeq T[1/f]$ and there is a canonical isomorphism
    $$\LL\eta_f(T)/f= \LL\eta_f(T)\otimes_{A}^{L} A/f\simeq (H^*(T/f), \beta_f),$$
    where $(H^*(T/f), \beta_f)$ is the Bockstein complex equal to $H^i( T\otimes_{A}^{\rm L} (f^iA/f^{i+1}A))$ in degree $i$, the differential being the boundary map associated to the triangle 
    $$ T\otimes_{A}^{\rm L} (f^{i+1}A/f^{i+2}A)\to T\otimes_A^{\rm L} (f^iA/f^{i+2}A)\to T\otimes_A^{\rm L} (f^iA/f^{i+1}A).$$
   This is discussed in \cite[Chapter 6]{BMS1} and  \cite[Lemma 5.9]{Bha}. 
    
    \item If $T\to L\to M$ is an exact triangle in $D(A)$, then 
    $\LL\eta_f(T)\to \LL\eta_f(L)\to \LL\eta_f(M)$ is also an exact triangle \textbf{if} the boundary map 
    $H^i(M/f)\to H^{i+1}(T/f)$ is the zero map for all $i$. For a regular element 
    $g\in A$ and $T\in D(A)$, the natural map 
    $\LL\eta_f(T)/g\to \LL\eta_f(T/g)$ is an isomorphism \textbf{if} $H^*(T/f)$ has no $g$-torsion. 
    
    \item If $I\subset A$ is a finitely generated ideal and $T\in D(A)$ is derived $I$-complete, then so is 
    $\LL\eta_f(T)$ \cite{BMS1}, \cite[Lemma 5.19]{Bha}.
    
    \item If $T\in D^{[0,d]}(A)$ and $H^0(T)$ is $f$-torsion-free then there are natural maps
    $\LL\eta_f(T)\to T$ and $T\to \LL\eta_f(T)$ whose compositions are 
    $f^d$. More precisely, if $T^{\bullet}$ is a representative concentrated in degrees $0,\ldots ,d$ and
    with $f$-torsion-free terms, then the first map is induced by $\eta_f(T^{\bullet})\subset T^{\bullet}$. Multiplication by $f^d$ on each of the two complexes factors over this inclusion map. When $T\in D^{\geq 0}(A)$, we will refer to 
    the map $\LL\eta_f(T)\to T$ as the \it{canonical map}.
     \end{enumerate}

  \subsection{The complexes $\ao$ and $\wt{\Omega}_{\calx}$} 
  \subsubsection{Fontaine rings.} Let 
   $$\O_{C}^{\flat}:=\varprojlim_{x\mapsto x^p} \O_{C}\simeq \varprojlim_{x\mapsto x^p} \O_{C}/p$$ be the tilt of 
   $\O_C$ (so that $C^{\flat}={\rm Frac}(\O_C^{\flat})$ is an algebraically closed field of characteristic $p$). 
   Let $A_{\rm inf}=W(\O_{C}^{\flat})$ and choose once and for all 
    a compatible sequence $(1,\zeta_p, \zeta_{p^2},...)$ of primitive $p$-power roots of $1$, giving rise to
    $\varepsilon=(1,\zeta_p, \zeta_{p^2},...)\in \O_{C}^{\flat}$. Letting $\varphi$ be the natural Frobenius automorphism of $A_{\rm inf}$, define
      $$\mu:=[\varepsilon]-1, \quad \xi:=\frac{\mu}{\varphi^{-1}(\mu)}=\frac{[\varepsilon]-1}{[\varepsilon^{1/p}]-1}\in A_{\rm inf}.$$
      The natural surjective map $\O_{C}^{\flat}\to \O_C/p$ lifts to a 
      map $\theta: A_{\rm inf}\to \O_C$ with kernel generated by $\xi$; the map $\theta$, in turn, lifts to a map $\theta_{\infty}: A_{\rm inf}\to W(\so_C)$ with kernel generated by $\mu$ (however, contrary to $\theta$, $\theta_{\infty}$ is not always surjective, see \cite[Lemma 3.23]{BMS1}). The kernel of the twisted map $\wt{\theta}:=\theta\phi^{-1}: A_{\rm inf}\to \O_C$
      is generated by 
        $$
      \tilde{\xi}:=\phi(\xi)=\frac{\phi(\mu)}{\mu}=\frac{[\varepsilon^p]-1}{[\varepsilon]-1}.
      $$
      We have $\tilde{\theta}(\mu)=\zeta_p-1$.
      
 We list the following properties \cite[2.25]{Bha} 
     \begin{enumerate}
     % \item The element $\xi$, resp.  $\wt{\xi}$, generates the kernel of $\theta$, resp. $\wt{\theta}$.
    % \item For $a,b\in \Z[1/p]$, if $v_p(a)\leq v_p(b)$ then $[\epsilon^a]-1$ divides $[\epsilon^b]-1$.
     \item %For $a\in \Z$, the image of $\frac{[\epsilon^a]-1}{[\epsilon]-1}$ in $A_{\rm inf}/\mu$ coincides with $a$. 
     
     $\tilde{\xi}$ modulo $\mu$ is equal to $p$.
     \item Since $A_{\rm inf}$ and its reduction mod $p$ are integral domains and 
     since $\xi, \tilde{\xi}, \mu$ are not $0$ modulo $p$, 
     $(p,\xi), (p,\tilde{\xi}), (p,\mu)$ are regular sequences, and so is the sequence
     $(\tilde{\xi}, \mu)$.
     
          \item The ideals $(p,\xi)$, $(p,\tilde{\xi})$, and $(\tilde{\xi},\mu)$  define the same topology on $A_{\rm \inf}$.
    % \item For $n\geq 1$, we have
    % $$
    % \mu=(\prod_{i=0}^{n-1}\phi^{-i}(\xi))\phi^{-1}(\mu).
    % $$ 
     \end{enumerate}

      The above  constructions naturally generalize to the case when $\O_C$ is replaced by a perfectoid ring.
         
              \subsubsection{Modified Tate twists.} The compatible sequence of roots of unity $(\zeta_{p^n})_{n}$ gives a trivialization
              $\zp(1)\simeq \zp$, and we will write $\zeta=(\zeta_{p^n})_{n}$ for the corresponding basis of $\zp(1)$.
                   By Fontaine's theorem \cite{Fo82}, the $\O_C$-module
    $$\O_C\{1\}:=T_p(\Omega^1_{\O_C/\zp})$$
 is free of rank $1$ and the natural map 
    ${\rm dlog}: \mu_{p^{\infty}}\to \Omega^1_{\O_C/\zp}$ induces an $\O_C$-linear injection 
    $${\rm dlog}: \O_C(1)\to \O_C\{1\}, \quad {\rm dlog}(\zeta)=({\rm dlog}(\zeta_{p^n}))_{n\geq 1}.$$
  The $\O_C$-module $\O_C\{1\}$ is generated by $$\omega:=\frac{1}{\zeta_p-1}{\rm dlog}(\zeta),$$ thus the annihilator of ${\rm coker}({\rm dlog})$ is $(\zeta_p-1)$. For any $\O_C$-module $M$, let $M\{1\}:=M\otimes_{\O_C} \O_C\{1\}$, and we will often write 
  $m\{1\}$ for the element of $M\{1\}$ corresponding to $m\in M$ (in particular,  $a\{1\}=a\cdot \omega$ in  $\O_C\{1\}$).

Finally, define 
$$A_{\rm inf}\{1\}:=\frac{1}{\mu}A_{\rm inf}(1),$$
and let $a\{1\}=\frac{1}{\mu}a(1)\in A_{\rm inf}(1)$, if $a\in A_{\rm inf}$.
The Frobenius $\varphi$ on $W(C^\flat)(1)$ induces an isomorphism
%and endow it with the natural Frobenius isomorphism\footnote{That is,
%$A_{\rm inf}\{1\}$ is a free $A_{\rm inf}$-module of rank $1$, with a basis $e=\frac{1}{\mu} \zeta$ such that 
%$\varphi(e)=\frac{\mu}{\varphi(\mu)}e$.}   
$$  \phi: A_{\rm inf}\{1\}\stackrel{\sim}{\to} (1/\tilde{\xi})A_{\rm inf}\{1\}.$$
Its inverse defines a map 
%induced by the Frobenius of $A_{\rm inf}$.    Define the map
$$
\phi^{-1}: A_{\rm inf}\{1\}\to A_{\rm inf}\{1\}.
$$
%as the restriction of its inverse to $A_{\rm inf}\{1\}$. 
There is a natural map 
$$\tilde{\theta}:=\theta\circ \varphi^{-1}: A_{\rm inf}\{1\}\to \O_C\{1\}$$
%sending $a\{1\},$  $a\in A_{\rm inf}$, to $\tilde{\theta}(a)\{1\}$ ($\zeta\mapsto \dlog \zeta$).
sending $a\{1\},$  for $a\in A_{\rm inf}$, to ${\theta}(\varphi^{-1}(a))\,\omega$.
     
If $M$ is an $A_{\rm inf}$-module, let 
$M\{i\}:=M\otimes_{A_{\rm inf}} A_{\rm inf}\{1\}^{\otimes i}$, $i\in\Z$.
The map
$\tilde{\theta}: A_{\rm inf}\{1\}\to \O_C\{1\}$ induces a map 
$\tilde{\theta}: M\{1\}\to (M/\tilde{\xi})\{1\}$ of $A_{\rm inf}$-modules (via the map
$A_{\rm inf}\to A_{\rm inf}/\tilde{\xi}$).

 \subsubsection{The complexes $\ao$ and $\wt{\Omega}_{\calx}$.}

   Let $\calx$ be a flat formal scheme
   over $\OC$, with smooth generic fibre $X$, seen as an adic space over $C$. There is a natural morphism of sites
   $$\nu: X_{\proet}\to \calx_{\eet},$$
  as well as a sheaf $\Ainf:=\mathbb{A}_{\rm inf, X}:=W(\varprojlim_{\varphi} \O_{X}^+/p)$ of $A_{\rm inf}$-modules on $X_{\proet}$, obtained by sheafifying Fontaine's construction $R\to A_{\rm inf}(R)$ on the basis of affinoid perfectoids of $X_{\proet}$. This sheaf is endowed with a bijective Frobenius $\varphi$ as well as with a surjective map 
   $$\theta: \Ainf\to \widehat{\O}^+_X:=\varprojlim \O_X^+/p^n,$$
   with kernel generated by the non-zero divisor $\xi$.

    Define 
    $$\ao:=\LL\eta_{\mu}(\R\nu_*\mathbb{A}_{\rm inf, X})\in D^{\geq 0}(\calx, A_{\rm inf}),$$
    a commutative algebra object, as well as 
     $$\wt{\Omega}_{\calx}:=\LL\eta_{\zeta_p-1}(\R\nu_* \widehat{\O}^+_X)
\in D(\calx_{\eet}),$$
     a commutative $\O_{\calx}$-algebra object in $D(\calx_{\eet})$.

  \subsection{The Hodge-Tate and de Rham specializations}
  
  \subsubsection{The smooth case.}
   
     Suppose first that $\calx$ is {smooth} over $\O_C$.
      The following result is proved in 
     \cite{BMS1} (for the Zariski site, but the proof is identical in our case).
     
     \begin{theorem}{\rm (Bhatt-Morrow-Scholze, \cite[Th. 8.3]{BMS1})}\label{BMS Omega} There is a natural isomorphism 
    of $\O_{\calx}$-modules on $\calx_{\eet}$
    $$H^i(\wt{\Omega}_{\calx})\simeq \Omega^i_{\calx/\O_C}\{-i\}.$$
      \end{theorem}
      
     We will recall the key relevant points since we will need 
some information about the construction of this isomorphism.
 
 Let $R$ be a formally smooth $\O_C$-algebra, such that ${\rm Spf}(R)$ is connected, together with an \'etale map
       $A:=\O_C\{T_1^{\pm 1},\ldots , T_d^{\pm 1}\}\to R$. We will simply say that $R$ is a {\it small algebra} and call the map $A\to R$ a {\it framing}. Let 
       $\widehat{\overline{R}}$ be the (perfectoid) completion of the normalization
       $\overline{R}$ of $R$ in the maximal pro-finite \'etale extension of $R[1/p]$, and let 
       $\Delta:={\rm Gal}(\overline{R}[1/p]/R[1/p])$. Define 
       $$A_{\infty}:=\O_C\{T_1^{\pm 1/p^{\infty}},\ldots , T_d^{\pm 1/p^{\infty}}\}, \quad R_{\infty}=R\widehat{\otimes}_A A_{\infty}.$$
       We have $\Gamma:={\rm Gal}(R_{\infty}/R)\simeq \zp(1)^d=\oplus_{i=1}^d \zp \gamma_i$, where 
       $\gamma_i$ sends $T_i^{1/p^n}$ to $\zeta_{p^n}T_i^{1/p^n}$ and 
fixes $T_j^{1/p^n}$ for $j\ne i$. 
By the almost purity theorem of Faltings, the 
   natural map (group cohomology is always continuous below) $\rg(\Gamma, R_{\infty})\to \rg(\Delta, \widehat{\overline{R}})$ is an almost quasi-isomorphism.
We have the following more precise results:

\begin{theorem}
{\rm (Bhatt-Morrow-Scholze, \cite[Cor. 8.13, proof of prop. 8.14]{BMS1})}
 Let $R$ be a small algebra together with a framing, as above, and let $X={\rm Sp}(R[1/p])$ and $\calx={\rm Spf}(R)$. 
 
{\rm a)} The natural maps
$$\LL\eta_{\zeta_p-1} \R\Gamma(\Gamma, R_{\infty})\to \LL\eta_{\zeta_p-1}\R\Gamma(X_{\proet}, \widehat{\O}_X^+)\to \R\Gamma(\calx, \wt{\Omega}_{\calx})$$
are quasi-isomorphisms. 

{\rm b)} Writing $\wt{\Omega}_R$ for any of these objects, the map
$\wt{\Omega}_R\otimes_R \so_{\calx}\to \wt{\Omega}_{\calx}$ is a quasi-isomorphism in $D(\calx_{\eet})$.

{\rm c)} If $R\to S$ is a formally \'etale map of small algebras, the natural map 
$\wt{\Omega}_R\otimes_{R}^{L} S\to \wt{\Omega}_{S}$
is a quasi-isomorphism. 

%d) If $R,S$ are small algebras, then the natural map
%$\tilde{\Omega}_{R_1}\widehat{\otimes}_{\O_C} \tilde{\Omega}_{R_2}\to \tilde{\Omega}_{R_1\widehat{\otimes}_{\O_C} R_2}$ is a quasi-isomorphism.
\end{theorem}

 Note that 
  $$H^i(\wt{\Omega}_R)\simeq H^i(\LL\eta_{\zeta_p-1} \rg(\Gamma, R_{\infty}))\simeq\frac{H^i(\Gamma, R_{\infty})}{H^i(\Gamma, R_{\infty})[\zeta_p-1]}\simeq 
%\frac{H^i(\Gamma, R)}{H^i(\Gamma, R)[\zeta_p-1]},$$
H^i(\Gamma,R),$$
the last isomorphism\footnote{Induced by the natural map $H^i(\Gamma, R)\to H^i(\Gamma, R_{\infty})$.} being a 
standard decompletion result (\cite[Prop. 8.9]{BMS1}).
  
The key result (not obvious since one needs to define the isomorphisms canonically, independent of coordinates!)  is then:
  
\begin{theorem}
{\rm (Bhatt-Morrow-Scholze, \cite[Chapter 8]{BMS1})}\label{tildeomega} Let $R$ be a small algebra. 
    
{\rm a)} There is a natural $R$-linear  isomorphism
$H^1(\wt{\Omega}_R)\simeq \Omega^1_{R/\O_C}\{-1\}$.

{\rm b)} The cup-product maps induce natural $R$-linear isomorphisms
$\wedge^iH^1(\wt{\Omega}_R)\simeq H^i(\wt{\Omega}_R)$
and hence  isomorphisms $H^i(\wt{\Omega}_R)\simeq \Omega^i_{R/\O_C}\{-i\}$.
\end{theorem}
  
     The isomorphism in a) is constructed  in \cite[Prop. 8.15]{BMS1} using completed cotangent complexes. 
     We will make it explicit, as follows: consider a framing $A\to R$ (recall that 
     $A=\O_C\{T_1^{\pm 1},\ldots , T_d^{\pm 1}\}$). By compatibility with base change from $A$ to $R$ of all objects involved, 
     it suffices to construct the isomorphism for $R=A$. Moreover we may  reduce to describing the isomorphism for 
     $A=\O_C\{T^{\pm 1}\}$, i.e., for $d=1$. Then the  twisted map
     $$\alpha: \Omega^1_{R/\O_C}\simeq H^1(\wt{\Omega}_R)\{1\}\simeq\frac{H^1(\Gamma, R_{\infty})}{H^1(\Gamma, R_{\infty})[\zeta_p-1]}\{1\}
     \xrightarrow{ x\mapsto (\zeta_p-1)x} (\zeta_p-1)H^1(\Gamma, R_{\infty})\{1\}$$
     is an isomorphism, described explicitly by
     $$\alpha\left(\frac{dT}{T}\right)=(\gamma\mapsto 1\otimes {\rm dlog}(\zeta_{\gamma}))=(\gamma\mapsto (\zeta_p-1)\otimes \frac{1}{\zeta_p-1}  {\rm dlog}(\zeta_{\gamma})),$$
     where $\zeta_{\gamma}=(\zeta_{\gamma,n})_n$, for $\gamma\in\Gamma$,  is defined by the formula $\zeta_{\gamma,n}:=\gamma(T^{1/p^n})/T^{1/p^n}$.
     
  \subsubsection{The semistable case.}
   Suppose now that $\calx$ is semistable. This means that, locally for the \'etale topology,
   $\calx={\rm Spf}(R)$, where $R$ admits an \'etale morphism of $\O_C$-algebras
   $$A:=\O_C\{T_0,\ldots ,T_r, T_{r+1}^{\pm 1},\ldots, T_d^{\pm 1}\}/(T_0\cdots T_r-p^q)\to R$$
   for some $d\geq 0$, $r\in \{0,1,\ldots ,d\}$ and some rational number $q>0$ (we fix once and for all an embedding 
   $p^{\Q}\subset C$). Equip $\O_C$ with the log-structure $\O_C\setminus \{0\}\to \O_C$ and 
   $\calx$ with the canonical log-structure, i.e. given by the sheafification of the subpresheaf 
   $\O_{\calx, \eet}\cap (\O_{\calx, \eet}[1/p])^{*}$ of $\O_{\calx, \eet}$. Let 
   $\Omega_{\calx/\O_C}$ be the finite locally free $\O_{\calx}$-module of logarithmic differentials on $\calx$. We have  the following result:
   
   \begin{theorem}
{\rm (\v{C}esna{v}i\v{c}ius-Koshikawa, \cite[Th. 4.2, Cor. 4.6, Prop. 4.8, Th. 4.11]{CK})} \label{CK}
   
{\rm a)} There is a natural $\O_{\calx, \eet}$-module isomorphism 
  $H^1(\wt{\Omega}_{\calx})\simeq \Omega^{1}_{\calx/\O_C}\{-1\}$ whose restriction to the smooth locus
   $\calx^{\rm sm}$ is the one given by Theorem \ref{BMS Omega}.

{\rm b)} The natural map $\wedge^i (H^1(\wt{\Omega}_{\calx}))\to H^i(\wt{\Omega}_{\calx})$ is an isomorphism and so there is a natural $\O_{\calx, \eet}$-module isomorphism
 $$H^i(\wt{\Omega}_{\calx})\simeq \Omega^{i}_{\calx/\O_C}\{-i\}.$$
    \end{theorem}

    \begin{remark} 1) A few words about a). The same arguments as in \cite{BMS1} (using completed cotangent complexes) give a map      $\Omega^{1,\cl}_{\calx/\O_C}\{-1\}\to \R^1\nu_*(\widehat{\O}_X^+)$, where we denoted by  the superscript $(-)^{\cl}$ the classical, non logarithmic, differential forms.
The results in \cite{BMS1} ensure that the resulting map (the second map being the natural projection)
     \begin{equation}
     \label{equation1}
     \Omega^{1,\cl}_{\calx/\O_C}\{-1\}\to \R^1\nu_*(\widehat{\O}_X^+)\to \frac{\R^1\nu_*(\widehat{\O}_X^+)}{\R^1\nu_*(\widehat{\O}_X^+)[\zeta_p-1]}\simeq H^1(\wt{\Omega}_{\calx})
     \end{equation}
     restricts to an isomorphism $\Omega^1_{\calx^{\rm sm}/\O_C}\{-1\}\simeq (\zeta_p-1)H^1(\wt{\Omega}_{\calx})|_{\calx^{\rm sm}}$. Moreover, one shows that $H^1(\wt{\Omega}_{\calx})$ is a vector bundle.
     Hence one can divide the map (\ref{equation1}) by $\zeta_p-1$
     to obtain a map $$
     \Omega^1_{\calx/\O_C}\{-1\}\to H^1(\wt{\Omega}_{\calx})$$
      which is an isomorphism over $\calx^{\rm sm}$. One shows that this extends to the isomorphism in a). 
     
2) A few words about the maps in b). Letting $K=\R\nu_*(\widehat{\O}_X^+)$ and using the 
      identifications
     $$H^1(\wt{\Omega}_{\calx})\simeq \frac{H^1(K)}{H^1(K)[\zeta_p-1]}, \,\, H^i(\wt{\Omega}_{\calx})\simeq \frac{H^i(K)}{H^i(K)[\zeta_p-1]},$$ they are induced by the product maps $H^1(K)^{\otimes i}\to H^i(K)$, which, in turn, are induced by the  product maps $$H^j(K)\otimes_{\O_{\calx, \eet}} H^k(K)\to H^{j+k}(K\otimes_{\O_{\calx, \eet}}^{L} K)\to H^{j+k}(K).$$

          \end{remark}

        We continue assuming that $\calx$ is semistable. Recall that 
         the map $\tilde{\theta}=\theta\circ \varphi^{-1}: \Ainf\to \widehat{\O}^+_X$ is surjective, with kernel generated by 
 $\tilde{\xi}=\varphi(\xi)$, and it sends $\mu$ to $\zeta_p-1$, inducing therefore a morphism
 $$\ao/\tilde{\xi}:=\ao\otimes^{L}_{A_{\rm inf}, \tilde{\theta}} \O_C\to \wt{\Omega}_{\calx}.$$
  
    \begin{theorem}\label{dRspec}{\rm(\v{C}esnavi\v{c}ius-Koshikawa, \cite[Th. 4.2, Th. 4.17, Cor. 4.6]{CK})}
  \begin{enumerate}
  \item  The  above morphism $\ao/\tilde{\xi}\to  \wt{\Omega}_{\calx}$ is a quasi-isomorphism.
  \item
  There is a natural quasi-isomorphism $\ao/\xi\stackrel{\sim}{\to} \Omega^{\bullet}_{\calx/\O_C}$.
  \item The complex $\ao$ is derived $\tilde{\xi}$-complete. Hence so is $\rg_{\eet}(\calx,\ao)$ (and  its cohomology groups). 
  \end{enumerate}
\end{theorem}
 For $i\geq 0$, (using the above theorems) we define
 \begin{itemize}
 \item  the {\em Hodge-Tate specialization} map
 \begin{enumerate}
 \item (on sheaves) as the composition $$\tilde{\iota}_{\htt}: \ao\to \ao/\tilde{\xi}\to \wt{\Omega}_{\calx}.$$
 \item  (on cohomology)
 $\iota_{\htt}: H^i_{\eet}(\calx,\ao)\to H^0_{\eet}(\calx, \Omega^i_{\calx/\so_C}\{-i\})$ as the  composition
 $$
 \iota_{\htt}: H^i_{\eet}(\calx,\ao)\lomapr{\tilde{\iota}_{\htt}}H^i_{\eet}(\calx,  \wt{\Omega}_{\calx})\to H^0_{\eet}(\calx, H^i(\wt{\Omega}_{\calx}))\simeq H^0_{\eet}(\calx, \Omega^i_{\calx/\so_C}\{-i\}),
 $$
 where the second map is the edge morphism in the spectral sequence
 $$
 E_2^{i,j}=H^i_{\eet}(\calx, H^j(\wt{\Omega}_{\calx})) \Rightarrow H^i_{\eet}(\calx,  \wt{\Omega}_{\calx}).
 $$
 \end{enumerate}
 \item 
the {\em de Rham specialization} map as the composition $\tilde{\iota}_{\dr}: \ao\to\ao/\xi\stackrel{\sim}{\to }\Omega^{\bullet}_{\calx/\mathcal{O}_C}$  yielding on cohomology 
 a map $$
 \iota_{\dr}: H^i_{\eet}(\calx,\ao)\lomapr{\tilde{\iota}_{\dr}} H^i_{\dr}(\calx).
$$
\end{itemize}
\subsection{$p$-adic nearby  cycles and $A_{\rm inf}$-cohomology}

   We review here a result from \cite{BMS2}, which describes integral $p$-adic \'etale cohomology in terms of 
   the complex $\ao$.    
      Let $X$ be a smooth adic space over $C$ and let $\calx$ be  a
      flat formal model of $X$ (not necessarily semistable !). 
      Fix an integer $i\geq 0$. Recall  that there is a natural endomorphism (as abelian sheaf, 
      \textbf{not} as $A_{\rm inf}$-module !) 
      $$\xi^i\varphi^{-1}: \tau_{\leq i}\ao\to \tau_{\leq i}\ao$$
       defined as the composition
      \begin{align*}
      \tau_{\leq i}\ao=\LL\eta_{\mu}\tau_{\leq i}\R\nu_*{\mathbb A}_{\inf}\lomapr{\phi^{-1}}\LL\eta_{\phi^{-1}(\mu)}\tau_{\leq i}\R\nu_*{\mathbb A}_{\inf}\stackrel{\xi^i}{\to} \LL\eta_{\xi} \LL\eta_{\phi^{-1}(\mu)}\tau_{\leq i}\R\nu_*\Ainf=\tau_{\leq i}\ao.
      \end{align*}

      The following result is proved in \cite{BMS2} in the good reduction case. As we show below the proof goes through  in a more general setting.
      \begin{theorem}\label{BMS2} {\rm(Bhatt-Morrow-Scholze, \cite[Chapter 10]{BMS2})}
         Let $X$ be a smooth adic space over $C$ with a  flat formal model $\calx$. Let $i\geq 0.$ There is a natural quasi-isomorphism
         $$\mu^i:  \tau_{\leq i}\R\nu_* \wh{\Z}_p\stackrel{\sim}{\to} \tau_{\leq i}[\tau_{\leq i}\ao\xrightarrow{1-\xi^i\varphi^{-1}} \tau_{\leq i}\ao],$$
         where    $\wh{\Z}_p:=\invlim_n\Z/p^n$ and $[\cdot]$ denotes the homotopy fiber.
      \end{theorem}
      \begin{remark}
      We warn the reader that this quasi-isomorphism is not Galois equivariant (in the case $X$ is defined over $K$). For an equivariant version see Corollary \ref{BMS22} below.
      \end{remark}
      \begin{proof}   We follow \cite{BMS2}  faithfully, but work directly on the $p$-adic level. Let $\psi_i=\xi^i\varphi^{-1}$, seen as an endomorphism of $\tau_{\leq i}\ao$ (as explained above) or of $T:=\R\nu_*\Ainf$ (defined in the obvious way). These two endomorphisms are compatible with 
      the canonical map
      ${\rm can}: \ao\to T$.
      
      We start with the following simple fact: 
      \begin{lemma} \label{AS} 
{\rm a)} For $i\geq j$, $1-\psi_i$ induces an automorphism of the sheaf $\Ainf/\mu^j$.
    
{\rm b)}
 There is an exact sequence of sheaves on $X_{\proet}$
        $$0\to \wh{\Z}_p\xrightarrow{\mu^i} \Ainf\xrightarrow{1-\psi_i} \Ainf\to 0.$$
      \end{lemma}
      
      \begin{proof}  a) This is proved by  Morrow in \cite[Lemma 3.5 (iii)]{Morr}.

      b) Clearly, this is a sequence. For surjectivity use a) to deduce surjectivity modulo  $\mu$ and then use the fact that $\Ainf$ is $\mu$-adically complete. 
     To describe $\ker(1-\psi_i)$, note that we have $\mu^i\wh{\Z}_p\subset \ker(1-\psi_i)$. It suffices thus to show that $\ker(1-\psi_i)\subset \mu^i\wh{\Z}_p$. But by part a) 
     we have 
     $$
     \ker(\Ainf\xrightarrow{1-\psi_i} \Ainf)\stackrel{\sim}{\leftarrow} \mu^i \ker(\Ainf\xrightarrow{1-\phi^{-1}} \Ainf)\stackrel{\sim}{\leftarrow}\mu^i\wh{\Z}_p.
     $$
     The last quasi-isomorphism follows from the Artin-Schreier 
       exact sequence \cite[Lemma 3.5 (ii)]{Morr}
      $$
      0\to \wh{\Z}_p\to \Ainf\xrightarrow{1-\varphi^{-1}} \Ainf\to 0.
      $$
     This finishes the proof of the lemma.
            \end{proof}
      
     Write $U^{\psi_i=1}$ for the homotopy fiber of $\psi_i-1: U\to U$ for $U\in \{\ao, T\}$.
               The above lemma gives an exact triangle
        $$ \R\nu_*\wh{\Z}_p\xrightarrow{\mu^i} T\xrightarrow{\psi_i-1} T,$$
        inducing a quasi-isomorphism
        $$\mu^i: \tau_{\leq i}  \R\nu_*\wh{\Z}_p\stackrel{\sim}{\to}  \tau_{\leq i} T^{\psi_i=1}.$$
        To finish the proof of our theorem, it remains to show (and this is the hard part) that the natural map (induced by the natural maps 
        ${\rm can}: \ao\to T$ and $\tau_{\leq i} \ao\to \ao$)
        $$\tau_{\leq i}(\tau_{\leq i}\ao)^{\psi_i=1} \to \tau_{\leq i}T^{\psi_i=1}$$
        is a quasi-isomorphism.
        
         By easy homological algebra, this happens if 
        $\psi_i$ acts bijectively on the kernel and cokernel of ${\rm can}_j: H^j(\ao)\to H^j(T)$ for $j<i$, bijectively on the kernel for $j=i$, and injectively on the cokernel for $j=i$. 
       This is clear for $j=0$: the map 
      ${\rm can}_0$ is then bijective as $H^0(\ao)=H^0(T)$ since $H^0(\ao)\simeq H^0(T)/H^0(T)[\mu]$ and 
      $H^0(T)[\mu]=0$ as $\Ainf$ is $\mu$-torsion-free. 
      
      Assume now that $j>0$. For $i\geq 0$, write $M_i=H^i(T)$.  Recall \cite[Lemma 6.4]{BMS1} that the map $\mu^j: M_j/M_j[\mu]\to H^j(\ao)$ is an isomorphism. It follows that, 
     for $0 < j \leq i$, the canonical map ${\rm can}_j$ fits into a natural exact sequence
      $$
      0\to M_j[\mu]\to M_j[\mu^j]\to H^j(\ao)\lomapr{{\rm can}_j} M_j\to M_j/\mu^j\to 0.
      $$
      This sequence is compatible with the operators $\psi_{i-j}$, $\psi_{i-j}$, $\psi_i$, $\psi_i$, $\psi_i$, respectively.
      Thus
        it suffices to show that $\psi_{i-j}$ is bijective on 
            $M_j[\mu^j]/M_j[\mu]$, that $\psi_i$ is bijective on $M_j/\mu^j$ for $j<i$, and  is injective for $j=i$. This follows from Lemma \ref{msri} below.          
        \end{proof}
        \begin{lemma}{\rm (\cite[Lemma 10.5]{BMS2})} \label{msri}   Let  $j\geq 1$, $i\geq 0$. 
        
         $\bullet$ $\psi_{l+j}$ is bijective on $M_i/\mu^j$ for $l>0$ and is  injective for $l=0$.
        
        $\bullet$ $\psi_l$ is bijective on $M_i[\mu^j]$ for $l>0$, surjective for $l=0$.
        
        $\bullet$  $\psi_l$ is bijective on 
       $M_i[\mu^j]/M_i[\mu]$,  for $l\geq 0$. 
            
         \end{lemma}   
            
            \begin{proof} We prove first that $\psi_l$ is injective on $M_i[\mu^j]$ for $l>0$. If $\psi_l(x)=0$ and $\mu^jx=0$, then $\psi_{l+1}(\mu x)=0$ and $\mu x\in M_i[\mu^{j-1}]$, thus, arguing by induction on $j$, we may assume that $j=1$. Suppose that $\mu x=0$ and $\psi_l(x)=0$, i.e., $x-\xi^l\varphi^{-1}(x)=0$. Since 
            $\xi\equiv p\pmod {\varphi^{-1}(\mu)}$ in $A_{\rm inf}$  and $\varphi^{-1}(\mu)$ kills $\varphi^{-1}(x)$, we deduce that $(1-p\xi^{l-1}\varphi^{-1})(x)=0$, which forces $x=0$, since $1-p\xi^{l-1}\varphi^{-1}$ is an automorphism of $\Ainf$ (as $\Ainf$ is $p$-adically complete), thus also of $T=\R\nu_* \Ainf$ and $M_i$. This proves the first step.
                             
Next, the commutative diagram of distinguished triangles
$$
 \xymatrix@R=.5cm@C=1cm{
 & T \ar[r]^-{\mu^j}\ar[d]^{\psi_{l}} & T \ar[r]\ar[d]^{\psi_{l+j}} & T/\mu^j\ar[d]^{\psi_{l+j}} & \\
 & T \ar[r]^-{\mu^j} & T \ar[r] & T/\mu^j& 
 }
$$
gives a commutative diagram
$$
 \xymatrix@R=.5cm{
 0\ar[r] & M_i/\mu^j \ar[r]\ar[d]^{\psi_{l+j}} & H^i(T/\mu^j) \ar[r]\ar[d]^{\psi_{l+j}} & M_{i+1}[\mu^j]\ar[r]\ar[d]^{\psi_l} & 0\\
 0\ar[r] & M_i/\mu^j \ar[r] & H^i(T/\mu^j) \ar[r] & M_{i+1}[\mu^j]\ar[r] & 0
 }
 $$
                    Since $\psi_{l+j}$ is bijective on $H^i(T/\mu^j)$ (Lemma \ref{AS} shows that 
                    $\psi_{l+j}$ is an automorphism of $T/\mu^j$), we deduce that 
                    $\psi_{l+j}$ is injective on $M_i/\mu^j$, $\psi_l$ is surjective on 
                    $M_{i+1}[\mu^j]$ and the cokernel of $\psi_{l+j}$ on $M_i/\mu^j$ identifies with the 
                    kernel of $\psi_l$ on $M_{i+1}[\mu^j]$. This last kernel is $0$ for $l>0$ (by the first step), thus 
                    $\psi_l$ is bijective on $M_{i+1}[\mu^j]$ (this holds trivially on $M_0[\mu^j]=0$) and $\psi_{l+j}$ is bijective on $M_i/\mu^j$ for $l>0$.
                             
          Finally,
           we need to show that $\psi_l$ is an automorphism of $M_i[\mu^j]/M_i[\mu]$. Surjectivity follows from that of $\psi_l$ on $M_i[\mu^j]$. For injectivity, note that if $\mu\psi_l(x)=0$, then $\psi_{l+1}(\mu x)=0$ and, since $\psi_{l+1}$ is injective on $M_i[\mu^{j-1}]$, we obtain $x\in M_i[\mu]$, as needed.

            \end{proof}
            Using the modified Tate  twists we can write the statement of the above theorem in the following way:
      \begin{corollary}\label{BMS22}
      Let $X$ be a smooth adic space over $C$ with a  flat formal model $\calx$. Let $i\geq 0$. There is a natural quasi-isomorphism
         $$\gamma:  \tau_{\leq i}\R\nu_* \wh{\Z}_p(i)\stackrel{\sim}{\to} \tau_{\leq i}[\tau_{\leq i}\ao\{i\}\xrightarrow{1-\varphi^{-1}} \tau_{\leq i}\ao\{i\}].$$
         It is Galois equivariant in the case $X$ is defined over $K$.
      \end{corollary}
      \begin{proof}Immediate from Theorem \ref{BMS2} and the following commutative diagram (which also defines the operator  $\varphi^{-1}$  on $\tau_{\leq i}\ao\{i\}$ in the  corollary):
      $$
\xymatrix@R=.7cm{
\tau_{\leq i} \R\nu_* \wh{\Z}_p\ar[r]^{\mu^i}\ar@{=}[d] &  \tau_{\leq i}\ao\ar[r]^{1-\xi^{i}\varphi^{-1}}\ar[d]^-{\mu^{-i}}_-{\wr}&  \tau_{\leq i}\ao\ar[d]^-{\mu^{-i}}_-{\wr}\\
\tau_{\leq i} \R\nu_* \wh{\Z}_p(i)\ar[r]^{\gamma=\can} &  \tau_{\leq i}\ao\{i\}\ar[r]^{1-\varphi^{-1}}& \tau_{\leq i} \ao\{i\}
}
      $$
      \end{proof}
      
\section{$A_{\rm inf}$-symbol maps}

       Let $X$ be a smooth adic space over $C$ and let $\calx$ be a flat $p$-adic formal model of 
       $X$ over $\O_C$. Let $\nu: X_{\proet}\to \calx_{\eet}$ be the natural map.

   \subsection{The construction of  symbol maps}   We will define compatible continuous pro-\'etale and $A_{\rm inf}$-symbol maps\footnote{We refer the reader to \cite[Sec. 2.2]{CDN3} for a discussion of topology on cohomologies of rigid analytic varieties and formal schemes. Integrally, we work in the category of pro-discrete modules, rationally -- in the category of locally convex topological vector spaces over $\Q_p$. But,  in this paper, we work with the naive topology on cohomology groups, i.e., the quotient topology, as opposed to the refined cohomology groups (denoted $\wt{H}$ in \cite{CDN3}) taken in the derived category of pro-discrete modules.} 
       
       \begin{equation}
       \label{symb1}
     r_{\proet}:  \O(X)^{*,\otimes i}\to H^i_{\proet}(X,\wh{\Z}_p(i)), \quad   r_{\rm inf}: \O(X)^{*,\otimes i}\to H^i_{\eet}(\calx, \ao\{i\}).
       \end{equation}
              
       For   $i=1$,
       we will construct below compatible  maps of sheaves
         \begin{equation}
         \label{symb2}
          c^{\proet}_1:\tau_{\leq 1} (\R\nu_*{\mathbb G}_m[-1])\to \tau_{\leq 1}\wh{\Z}_p(1),\quad c^{\rm inf}_1:\tau_{\leq 1} (\R\nu_*{\mathbb G}_m[-1])\to \tau_{\leq 1}\ao\{1\}.
         \end{equation}
         Applying $H^1_{\eet}(\calx,-)$ and observing 
         that $$H^1_{\eet}(\calx, \tau_{\leq 1} (\R\nu_*{\mathbb G}_m[-1]))\stackrel{\sim}{\to}H^1_{\eet}(\calx, \R\nu_*{\mathbb G}_m[-1])=H^0_{\eet}(\calx, \R\nu_*{\mathbb G}_m)\simeq \O(X)^{*},$$
we get that the maps $c^{\proet}_1, c^{\rm inf}_1$  induce  global symbol maps 
          $$r_{\proet}: \O(X)^{*}\to H^1_{\proet}(X, \wh{\Z}_{p}(1)),\quad r_{\rm inf}: \O(X)^{*}\to H^1_{\eet}(\calx, \ao\{1\}).$$
           For $i\geq 1$, we define the symbol maps (\ref{symb1}) using cup product: $x_1\otimes\cdots\otimes 
          x_i\mapsto r_*(x_1)\cup\cdots\cup r_*(x_i).$ 
         
        The construction of the  first map  in (\ref{symb2})  uses
        the Kummer exact sequence 
          on $X_{\proet}$
         $$0\to \wh{\Z}_p(1)\to \varprojlim_{x\mapsto x^p} {\mathbb G}_m\to {\mathbb G}_m\to 0,$$
       inducing, by projection to $\calx_{\eet}$, 
         the Chern class map
         $$c_1^{\proet}: \R\nu_*{\mathbb G}_m[-1]\to \R\nu_*\wh{\Z}_p(1).$$
         
          The construction of the second map in (\ref{symb2}) uses the above Kummer exact sequence and 
  the {\it twisted} Artin-Schreier exact sequence\footnote{
  Note that, for $x\in \Ainf$,
 $$(1-\varphi^{-1})(x\{1\})=[(1-\xi \varphi^{-1})(x)]\{1\}$$
 and this is $0$ precisely when $x=\mu y$ with $y\in \wh{\Z}_p$. }
 on $X_{\proet}$
 $$0\to \wh{\Z}_p(1)\lomapr{\gamma}\Ainf\{1\}\xrightarrow{1-\varphi^{-1}} \Ainf\{1\}\to 0,
 $$
 where the map $\gamma$  is defined by $x(1)\mapsto \mu x\{1\}, x\in \wh{\Z}_p.$
 Pushing down to $\calx_{\eet}$ we obtain a map
 $$\gamma: \tau_{\leq 1} \R\nu_*\wh{\Z}_p(1)\to \tau_{\leq 1} \R\nu_*\Ainf\{1\}.$$
 On the other hand, Corollary \ref{BMS22} gives us a natural map
 $${\gamma}: \tau_{\leq 1} \R\nu_*\wh{\Z}_p(1)\to \tau_{\leq 1}\ao\{1\}.$$
 It is easy to check that the above two maps are compatible via 
 the canonical map
 $\tau_{\leq 1}\ao\{1\}\to \tau_{\leq 1} \R\nu_*\Ainf\{1\}$ and that  the composition
 $$
 \tau_{\leq 1} \R\nu_*\wh{\Z}_p(1)\lomapr{\gamma} \tau_{\leq 1}\ao\{1\}\lomapr{1-\phi^{-1}}\tau_{\leq 1}\ao\{1\}
 $$
 is the zero map.

 Now, our symbol map is simply the composition
  $$c^{\rm inf}_1:\tau_{\leq 1}\R\nu_*{\mathbb G}_m [-1]\xrightarrow{c_1^{\proet}} \tau_{\leq 1}\R\nu_*\wh{\Z}_p(1)\xrightarrow{\gamma}
  \tau_{\leq 1}\ao\{1\}.$$

  \subsection{Compatibility with the Hodge-Tate symbol map}\label{Hodge-tate}Let $\calx_{\so_K}$ be a semistable formal schemes over $\so_K$. Let $M$ be the sheaf of monoids on $\calx_{\so_K}$ defining the log-structure, $M^{\gp}$ its group completion. This log-structure is canonical, in the terminology of Berkovich \cite[2.3]{BerL}, i.e., $M(U)=\{x\in\so_{\calx_{\so_K}}(U)| x_K\in\so^{*}_{\calx_K}(U_K)\}$. This is shown in \cite[Th. 2.3.1]{BerL}, \cite[Th. 5.3]{BerN} and applies also to semistable formal schemes with self-intersections. It follows that  $M^{\gp}(U)=\so^{*}_{\calx_K}(U_K)$. Set $X_K:=\calx_{\so_K,K}, \calx:=\calx_{\so_C}, X:=X_{K,C}$.

  For $i\geq 1$, the Hodge-Tate symbol maps $$
  r_{\htt}:\quad  \so(X_K)^{*, \otimes i}\to H^0_{\eet}(\calx,\Omega^i_{\calx})
  $$
  are defined by taking cup products of the Chern class maps
  $$
  c_1^{\htt}:\quad \so(X_K)^*\to \tau_{\leq 1}(\R\nu_*{\mathbb G}_m[-1])\to \Omega_{\calx}[-1],\quad x\mapsto \dlog (x).
  $$
  The purpose of this section is to prove the following fact:
  \begin{proposition}
  \label{compatibility2}Let $i\geq 1$. 
  The symbol maps $r_{\rm inf}$ and $r_{\htt}$ are compatible under the Hodge-Tate specialization map $\iota_{\htt}$, i.e., $\iota_{\htt}r_{\rm inf}{|\so(X_K)^{*,\otimes i}}=r_{\htt}$.
  \end{proposition}
  \begin{proof}
  {\em The case $i=1$.}    
      Consider  the composition $\iota_{\htt}c_1^{\rm inf}$:
             $$\O(X_K)^*\to \O(X)^*\simeq H^1_{\eet}(\calx, \tau_{\leq 1} (\R\nu_*{\mathbb G}_m[-1]))\lomapr{c_1^{\rm inf}}H^1_{\eet}(\calx, \ao\{1\})\lomapr{\iota_{\htt}} H^0_{\eet}(\calx, \Omega^1_{\calx}).$$
   We need to show that 
   \begin{lemma}
   \label{compatibility1}
    The above  composition $\iota_{\htt}c_1^{\rm inf}$  is equal to the map $c_1^{\htt}={\rm dlog}: \O(X_K)^*\to H^0_{\eet}(\calx, \Omega^1_{\calx})$.
   \end{lemma}
   \begin{proof} From the definitions, it suffices to check that 
the composition 
 $$\O_X^{*}\lomapr{} \R^1\nu_*\wh{\Z}_p(1)\lomapr{\gamma} H^1(\ao\{1\})\lomapr{\tilde{\theta}} H^1(\wt{\Omega}_{\calx}\{1\})\simeq \Omega^1_{\calx}$$ is
 the map ${\rm dlog}$. To do this we 
 will study the diagram

      $$
 \xymatrix{
 & & \frac{\R^1\nu_*\Ainf}{\R^1\nu_*\Ainf[\mu]}\{1\} \ar[r]^{\tilde{\theta}}\ar[d]^{\wr}_{\gamma_1}&  \frac{\R^1\nu_*\widehat{\O}^+}{\R^1\nu_*\widehat{\O}^+[\zeta_p-1]}\{1\} \ar[d]^{\wr}_{\gamma_2} & \\
\O_X^{*}\ar[r]^-{c_1^{\proet}} & \R^1\nu_*\wh{\Z}_p(1) \ar[r]^-{\gamma} \ar[rd]^{\beta} & H^1(\ao\{1\}) \ar[r]^{\tilde{\theta}} \ar[d]^{\rm can} & H^1(\wt{\Omega}_{\calx}\{1\}) \ar[r]^-{\sim}_-{\alpha_1} \ar[d]^{\rm can} & \Omega^1_{\calx} \ar[rd]^{\iota} \\
& & \R^1\nu_*\Ainf\{1\} \ar[r]^{\tilde{\theta}} & \R^1\nu_*\widehat{\O}^+\{1\} \ar[r] & \R^1\nu_*\widehat{\O}(1) \ar[r]^-{ \sim}_-{\alpha_2} & \varepsilon_*\Omega^1_X
 }
 $$
  Here $\varepsilon: X_{\eet}\to \calx_{\eet}$ is the canonical map. 
   A few words about this diagram:

   $\bullet$ the map $ \R^1\nu_*\widehat{\O}^+\{1\}\to 
\R^1\nu_*\widehat{\O}(1)$ is induced by the natural map $\widehat{\O}^+\to \widehat{\O}$ and the map
$\O_C\{1\}\to C(1)$ is given by $x\{1\}\to \frac{x}{\zeta_p-1}(1)$. 

   $\bullet$ let $\nu_X: X_{\proet}\to X_{\eet}$ be the canonical projection. There is an isomorphism
   \begin{equation}
   \label{Scholze1}
   \alpha_2:\R^1\nu_{X,*}\widehat{\O}(1)\stackrel{\sim}{\to} \Omega^1_X
   \end{equation}
    due to Scholze (\cite[Lemma 3.24]{Surv}). It is  uniquely characterized by the property that its inverse 
   is the unique $\O_X$-linear map $\alpha^{-1}_2: \Omega^1_X\to \R^1\nu_{X,*}\widehat{\O}(1)$ making the following diagram commute
   \begin{equation}
   \label{diag22}
   \xymatrix@R=.6cm{
     \O_X^{*} \ar[r]^-{c_1^{\proet}} \ar[d]_{\rm dlog} &  \R^1\nu_{X,*}\wh{\Z}_p(1)  \ar[d]\\
            \Omega^1_{X}\ar[r]^-{\alpha^{-1}_2} & \R^1\nu_{X,*}\widehat{\O}(1)
      },
      \end{equation}
      the right vertical map being the obvious one.
      
      $\bullet$ The map $\alpha_2:  \R^1\nu_*\wh{\so}(1)\to \varepsilon_*\Omega^1_X$ is defined as the composition
      \begin{equation}
      \label{lyon2}
      \R^1\nu_*\wh{\so}(1)\simeq  \R^1(\varepsilon \nu_X)_*\wh{\so}(1)\stackrel{\sim}{\to}\varepsilon_*\R^1\nu_{X,*}\wh{\so}(1)\simeq \varepsilon_*\Omega^1_X,
      \end{equation}
      where the last map is the isomorphism (\ref{Scholze1}). Recall that this last  isomorphism generalizes to isomorphisms  \cite[Prop. 3.23]{Surv}:
      $$
      \R^i\nu_{X,*}\widehat{\O}(i)\simeq \Omega^i_X, \quad i\geq 0.
      $$
      This combined with Tate's acyclicity theorem allows us to infer that the second map in (\ref{lyon2}) is an isomorphism.

      We claim that the composite 
      $$ \O_X^{*}\to \R^1\nu_{X,*}\wh{\Z}_p(1) \to \R^1\nu_{X,*}\Ainf\{1\} \to \R^1\nu_{X,*}\widehat{\O}^+\{1\}\to 
      \R^1\nu_{X,*}\widehat{\O}(1)\to  \Omega^1_X$$
      is the ${\rm dlog}$ map. Using the characterization of Scholze's isomorphism (\ref{Scholze1}), this comes down to checking that the map $\R^1\nu_{X,*}\wh{\Z}_p(1) \to \R^1\nu_{X,*}\Ainf\{1\} \to \R^1\nu_{X,*}\widehat{\O}^+\{1\}\to 
      \R^1\nu_{X,*}\widehat{\O}(1)$ is the obvious one. But, by construction, this map  is induced by the map
      $$\wh{\Z}_p(1)\to \Ainf\{1\}\to \widehat{\O}^+\{1\}\to \widehat{\O}(1)$$
      sending $x(1)$ to $x\mu\{1\}$, then to $\tilde{\theta}(x\mu\{1\})=x(\zeta_p-1)\{1\}$, then to 
      $x(1)$, as wanted.
      
      To finish the proof of the lemma, it remains to check that the above big diagram commutes. The only nonobvious commutativity is that of 
the right-bottom trapezoid, i.e., we need to check the compatibility of the maps $\alpha_1$ and $\alpha_2$. Call 
       
       \begin{equation}
       \label{rho1}\rho: \Omega^1_{\calx} \lomapr{\alpha_1^{-1}} H^1(\wt{\Omega}_{\mathfrak{X}}\{1\}) \xrightarrow{\rm can}
      \R^1\nu_*\widehat{\O}^+\{1\}\to  \R^1\nu_*\widehat{\O}(1)\lomapr{\alpha_2} \varepsilon_*\Omega^1_X .
      \end{equation}
    We want to show that $\rho=\iota$. It suffices to check this on the smooth locus of $\calx$, which reduces us to the case when $\calx$ is smooth. Note that the maps $\rho, \iota$ are $\so_{\calx}$-linear. This is clear for $\iota$; for $\rho$ we look at the individual maps in the composition (\ref{rho1}) that defines it: the second and the third map are clearly $\so_{\calx}$-linear, for the first map we use Theorem \ref{CK}, and for the last map we use the discussion above.
    Now, the claim that $\rho=\iota$ is  local, so we way assume that $\calx$ is associated to a small algebra $R$ with a framing $A=\O_C\{T_i^{\pm 1}\}\to R$. By functoriality, we may  reduce to the case when $R=A$ and  $A=\O_C\{T^{\pm 1}\}$. 
    
    Now, the desired compatibility follows from the very construction of the  isomorphism $\alpha_1$. More precisely, 
     since ${\rm can}\circ \gamma_2$ is the multiplication by $\zeta_p-1$, we have 
     $$(\zeta_p-1)\gamma_2^{-1}(\alpha_1^{-1}(dT/T))={\rm can}(\alpha_1^{-1}(dT/T)).$$
     As we have already seen (cf. discussion after Theorem \ref{tildeomega}) this corresponds to the element $(\gamma\mapsto (\zeta_p-1)\otimes \frac{1}{\zeta_p-1}\dlog(\zeta_{\gamma}))$ in $(\zeta_p-1)H^1(\Gamma, A_{\infty})\{1\}$. Now the compatibility of the map $\alpha_2$ with the Kummer map (see the diagram (\ref{diag22})) shows that $\rho(dT/T)=dT/T$, as wanted.
            \end{proof}

{\em The case $i\geq 1$.}   
Take now  the symbol maps
     $$ r_{\rm inf}: \O(X)^{*, \otimes i}\to H^i_{\eet}(\calx, \ao\{i\})$$
and consider  the composition $\iota_{\htt}r_{\rm inf}$:
     $$\O(X)^{*, \otimes i}\to H^i_{\eet}(\calx, \ao\{i\})\to H^i_{\eet}(\calx, \wt{\Omega}_{\calx}\{i\})\to H^0_{\eet}(\calx, H^i( \wt{\Omega}_{\calx}\{i\}))\simeq
     H^0_{\eet}(\calx, \Omega^{i}_{\calx}).$$
     To finish the proof of our proposition, in view of Lemma \ref{compatibility1},  it suffices to check that this composition is compatible with products. But, the third map in the composition is clearly compatible with products and  the first map is compatible with products by  definition.  
    The second  map is induced by the map $\tilde{\theta}: \ao\to \wt{\Omega}_{\calx}$ hence it is also compatible with products. Finally,
    the last map is the isomorphism given by Theorem \ref{CK} hence is compatible with products by its very definition. 
    \end{proof}     
  
  \section{The $A_{\rm inf}$-cohomology of   Drinfeld symmetric spaces}\label{drinfeld1}
     
    Let $\mathcal{H}=\mathbb{P}((K^{d+1})^*)\simeq\mathbb{P}^d(K)$ be the space of $K$-rational hyperplanes in $K^{d+1}$. Let 
          $${\mathbb H}^d_K:={\mathbb P}^d_K\setminus \cup_{H\in \mathcal{H}} H$$ 
          be the Drinfeld symmetric space of dimension $d$. It is a rigid analytic space. Let 
      $\calx_{\O_K}$ be the standard semistable   formal model over $\O_K$ of $\mathbb{H}_K^d$ (see \cite[Section 6.1]{GKF}). Let $\calx:=\calx_{\O_K}\widehat{\otimes}_{\O_K} \O_C$,  let $X:= \mathbb{H}_K^d\hat{\otimes}_{K} C$ be the rigid analytic generic fiber of $\calx$, and let $X_K=\mathbb{H}_K^d$. The group $G={\rm GL}_{d+1}(K)$ acts naturally on all these objects.

The main goal of this section is to prove the following (here and elsewhere in the paper, the completed tensor product is taken in the category of pro-dicrete modules):      
      \begin{theorem}\label{main}Let $i\geq 0$. 
                 There is a $G\times \sg_K$-equivariant isomorphism of topological $A_{\rm inf}$-modules
                 $$ A_{\rm inf}\widehat{\otimes}_{\zp} {\rm Sp}_i(\zp)^*\simeq H^i_{\eet}(\calx, \ao\{i\}),$$      
                 where ${\rm Sp}_i(\zp)^*$ is the $\zp$-dual of a generalized Steinberg representation (see Section \ref{Steinberg} for a definition). This isomorphism is compatible with the operator $\phi^{-1}$.
                                  \end{theorem}

    \subsection{Generalized Steinberg representations and their duals} \label{Steinberg}
       \subsubsection{Generalized Steinberg representations.}
     Let $B$ be the upper triangular Borel subgroup of $G$ and 
   $\Delta=\{1,2,\ldots ,d\}$, identified with the set of simple roots associated to $B$.
      For each subset 
   $J$ of $\Delta$ we let
   $P_J$ be the corresponding parabolic subgroup of $G$ and set $X_J=G/P_J$, a compact topological space. %having a Bruhat decomposition
   %$$X_J=\coprod_{w\in W/W_J} C_J(w)=\coprod_{w\in W^J} C_J(w)$$
   %into cells 
    %$C_J(w)=BwP_J/P_J$. If $\Delta\setminus J=\{i_0<...<i_r\}$, $X_J$ classifies flags
    %$W_0\subset W_1\subset...\subset W_r$ of $K^{d+1}$ such that $\dim_K W_j=i_j$ for 
    %$0\leq j\leq r$, by sending $gP_J$ to the flag $g(\sum_{i=1}^{i_j} Ke_i)_{0\leq j\leq r}$. 
    
    If $A$ is an abelian group and $J\subset \Delta$, let 
    $${\rm Sp}_J(A)=\frac{{\rm LC}(X_J, A)}{\sum_{i\in \Delta\setminus J} {\rm LC}(X_{J\cup \{i\}}, A)},$$
       where ${\rm LC}$ means locally constant (automatically with compact support). This is 
   a smooth $G$-module over $A$ and we have a canonical isomorphism 
   ${\rm Sp}_J(A)\simeq{\rm Sp}_J(\mathbf{Z})\otimes A$. For $J=\emptyset$ we obtain the usual Steinberg representation with coefficients in $A$, while for $J=\Delta$ we have ${\rm Sp}_J(A)=A$. For $r\in \{0,1,\ldots,d\}$ we use the simpler notation 
        $${\rm Sp}_r:={\rm Sp}_{\{1,2,\ldots ,d-r\}}.$$
             
             We will need the following result: 
        
  \begin{theorem}{\rm (Grosse-Kl\"onne, \cite[Cor. 4.3]{GKir})}\label{GK irred}
   If $A$ is a field of characteristic $p$ then 
 ${\rm Sp}_J(A)$ (for varying~$J$) are the irreducible constituents of ${\rm LC}(G/B, A)$, each occurring with multiplicity $1$.    
  \end{theorem}

    \subsubsection{Duals of generalized Steinberg representations.}       \label{duals}
          
         If $\Lambda$ is a topological ring, then ${\rm Sp}_J(\Lambda)$ has a natural topology: the space 
         $X_J$ being profinite, we can write $X_J=\varprojlim_{n} X_{n,J}$ for  finite sets 
         $X_{n,J}$ and then 
          ${\rm LC}(X_J, \Lambda)=\varinjlim_{n} {\rm LC}(X_{n,J}, \Lambda)$, each 
          ${\rm LC}(X_{n,J}, \Lambda)$ being a finite free $\Lambda$-module endowed with the natural topology.       
    
                    Let $M^*:={\rm Hom}_{\rm cont}(M,\Lambda)$ for any topological $\Lambda$-module $M$, and equip $M^*$ with the weak topology.
          Then ${\rm Sp}_J(\Lambda)^*$ is naturally isomorphic to $\varprojlim_{n} {\rm LC}(X_{n,J}, \Lambda)^*$, i.e. a countable inverse limit of finite free $\Lambda$-modules.
          In particular, suppose that $L$ is a finite extension of $\Q_p$. Then 
        ${\rm Sp}_J(\so_L)^*$ is a compact $\so_L$-module, which is torsion-free.

     If $S$ is a profinite set and $A$ an abelian group, let $$D(S,A)={\rm Hom}({\rm LC}(S,\mathbf{Z}), A)={\rm LC}(S,A)^*$$ be the space of $A$-valued locally constant distributions on 
         $S$. We recall the interpretation of ${\rm Sp}_i(\zp)^*$ in terms of distributions. 
    Recall that $\mathcal{H}$ denotes the compact space of $K$-rational hyperplanes in $K^{d+1}$.  
    If $H\in \mathcal{H}$, let $\ell_H$ be a unimodular equation for $H$ (thus $\ell_H$ is a linear form with 
    integer coefficients, at least one of them being a unit).   
     Let ${\rm LC}^{c}(\mathcal{H}^{i+1}, \mathbf{Z})$ be the space of locally constant functions 
       $f: \mathcal{H}^{i+1}\to \mathbf{Z}$ such that, for all $H_0,...,H_{i+1}\in \mathcal{H}$,
       $$f(H_1,...,H_{i+1})-f(H_0, H_2,..., H_{i+1})+\cdots +(-1)^{i+1}f(H_0,...,H_{i})=0$$
       and, moreover, if 
       $\ell_{H_j}$, $0\leq j\leq i$, are linearly dependent, then     
       $f(H_0,...,H_i)=0$. 
     The work of Schneider-Stuhler \cite{SS} gives a $G$-equivariant isomorphism
 $${\rm Sp}_i(\mathbf{Z})\simeq {\rm LC}^{c}(\mathcal{H}^{i+1}, \mathbf{Z}).$$
     It follows that  the inclusion ${\rm LC}^{c}(\mathcal{H}^{i+1}, \mathbf{Z})\subset {\rm LC}(\mathcal{H}^{i+1}, \mathbf{Z})$ gives rise to a strict exact sequence
  \begin{equation}
  \label{strict1}0\to D(\mathcal{H}^{i+1}, A)_{\rm deg}\to D(\mathcal{H}^{i+1}, A)\to {\rm Hom}({\rm Sp}_i(\mathbf{Z}), A)\to 0,
  \end{equation}
        where  $D(\mathcal{H}^{i+1}, A)_{\rm deg}$ is the space of degenerate distributions (which is defined via the exact sequence above).

   \subsection{Integral de Rham cohomology of Drinfeld symmetric spaces}

   Recall the following 
      acyclicity result of Grosse-Kl\"onne, which played a crucial role in \cite{CDN3}. 
 
 \begin{theorem}{\rm (Grosse-Kl\"onne,  \cite[Th. 4.5]{GK1}, \cite[Prop. 4.5]{GK2})}\label{GK} For $i>0$,   $j\geq 0$,
  we have $H^i_{\eet}(\calx, \Omega^j_{\calx})=0$ and 
  $d=0$ on $H^0_{\eet}(\calx, \Omega^j_{\calx})$. In particular, we have a natural quasi-isomorphism
  $$ \rg_{\dr}(\calx)\simeq\R\Gamma_{\eet}(\calx, \Omega^{\bullet}_{\calx})\simeq \bigoplus_{i\geq 0} \Gamma_{\eet}(\calx, \Omega^i_{\calx})[-i].$$
 \end{theorem}
Using it and some extra work, we have obtained the following description of $H^i_{\dr}(\calx_{\so_K})$: 
   
   \begin{theorem}{\rm (Colmez-Dospinescu-Nizio\l, \cite[Th. 6.26]{CDN3})}\label{ours}
   There are natural de Rham and Hodge-Tate regulator maps 
   \begin{align*}
 &  r_{\dr}: D(\sh^{i+1},\so_K)\to H^i_{\dr}(\calx_{\so_K}),\\
   & r_{\htt}: D(\sh^{i+1},\so_K)\to H^0_{\htt}(\calx_{\so_K},\Omega^i_{\calx_{\so_K}})
   \end{align*} that induce  topological $G\times\sg_K$-equivariant isomorphisms in the commutative diagram:
   $$
   \xymatrix{
    {\rm Sp}_i(\O_K)^*\ar[dr]^{\sim}_{r_{\htt}}\ar[r]^{\sim}_{r_{\dr}} & H^i_{\dr}(\calx_{\so_K})\\
 & H^0(\calx_{\so_K},\Omega^i_{\calx_{\so_K}})\ar[u]^{\wr}
 }
   $$
   \end{theorem}
 \begin{proof}{\rm (Sketch)} Our starting point was the computation of Schneider-Stuhler \cite{SS}: a $G$-equivariant topological isomorphism
 $$
 \xymatrix{
{\rm Sp}_r(K)^*\ar[r]^-{\alpha_S}_-{\sim} & H^i_{\dr}(X_K).
}
 $$
 Iovita-Spiess \cite{IS} made this isomorphism explicit: they have proved that there is a commutative diagram
 $$
 \xymatrix{
 0\ar[r] & D(\mathcal{H}^{i+1}, K)_{\rm deg}\ar[r] & D(\mathcal{H}^{i+1}, K)\ar[r] \ar@{->>}[dr]^{r_{\dr}} &  {\rm Sp}_i(K)^*\ar[r] \ar[d]^{\alpha_S}_{\wr}& 0\\
 & & &H^i_{\dr}(X_K)
 }
 $$
With a help from a detailed analysis of the integral Hyodo-Kato cohomology of the special fiber of $\calx_{\so_K}$ and some representation theory\footnote{We used here two facts: (a) ${\rm Sp}_i(\so_K)$ is, up to a $K^*$-homothety, the unique $G$-stable lattice in ${\rm Sp}_i(K)$; (b) 
${\rm Sp}_i(k)$ is irreducible.} this computation can be lifted to $\so_K$.  
 \end{proof}
 
  The following computation follows immediately: 
   \begin{corollary}\label{integraldR}
   \begin{enumerate}
   \item 
   The de Rham regulator $r_{\dr}$ induces a  topological $G$-equivariant isomorphism 
   $$
   r_{\dr}:  {\rm Sp}_i(\O_K)^* \wh{\otimes}_{\so_K}\so_C\stackrel{\sim}{\to} H^i_{\dr}(\calx_{\so_K}) \wh{\otimes}_{\so_K}\so_C\stackrel{\sim}{\to}H^i_{\dr}(\calx).
   $$
\item The Hodge-Tate  regulator $r_{\htt}$ induces a  topological $G$-equivariant isomorphism 
   $$
   r_{\htt}:  {\rm Sp}_i(\O_K)^* \wh{\otimes}_{\so_K}\so_C\stackrel{\sim}{\to} H^0_{\eet}(\calx_{\so_K},\Omega^i_{\calx_{\so_K}}) \wh{\otimes}_{\so_K}\so_C\stackrel{\sim}{\to}H^0_{\eet}(\calx,\Omega^i_{\calx}).
   $$
\end{enumerate}
\end{corollary}
 \subsection{Integrating  symbols} \label{deszcz1}  
          Let $i\geq 1$.  In this section, our goal is to construct natural compatible regulator maps 
    \begin{align*}
     r_{\eet}:  {\rm Sp}_i(\zp)^*\to H^i_{\eet}(X, \Z_p(i)),\quad      r_{\rm inf}:  {\rm Sp}_i(\zp)^*\to H^i_{\eet}(\calx, \ao\{i\})
     \end{align*}
      that are compatible with the classical \'etale and $A_{\rm inf}$-regulators.
We will show later that (the linearizations of) both regulators are $G\times \sg_K$-equivariant  isomorphisms.
 The maps $r_{\eet}$, $r_{\rm inf}$ are constructed by interpreting elements of ${\rm Sp}_i(\zp)^*$ as suitable distributions (see  the  discussion in Section \ref{duals}), and integrating \'etale and $A_{\rm inf}$-symbols of invertible functions on $\mathbb{H}^d_K$ against them. This idea appears in Iovita-Spiess \cite{IS} and was also heavily used in \cite{CDN3}.

           \subsubsection{Integrating \'etale symbols} We start with the construction of the \'etale regulator map 
    $$ r_{\eet}:  {\rm Sp}_i(\zp)^*\to H^i_{\eet}(X, \Z_p(i)).$$
         Fix a cohomological degree $i$ and set 
        $M: =H^i_{\eet}(X,\Z_p(i)).$
        For $H_0,...,H_i\in \mathcal{H}$, let 
        $$\psi_{\eet}(H_0,...,H_i):=r_{\eet}\left(\frac{\ell_{H_1}}{\ell_{H_0}}\otimes...\otimes \frac{\ell_{H_i}}{\ell_{H_0}}\right)\in M,$$
        where $r_{\eet}: \O(\mathbb{H}_C^d)^{*, \otimes i}\to M$ is the \'etale regulator map. 
          It is clear that this definition is independent of the choice of the unimodular equations for $H_0,...,H_i$.

\begin{proposition}
\label{bursztyn1}Let $i\geq 1$. 
\begin{enumerate}
\item Let $\delta_x$ denote the Dirac distribution at $x$.
 There is a unique continuous $\Z_p$-linear map 
 $$r_{\eet}:  D(\mathcal{H}^{i+1}, \zp)\to H^i_{\eet}(X, \Z_p(i))$$
 such that $r_{\eet}(\delta_{(H_0,...,H_i)})=\psi_{\eet}(H_0,...,H_i)$ for all 
 $H_0,...,H_i\in \mathcal{H}$.
 
  \item The map $r_{\eet}$ factors through the quotient ${\rm Sp}_i(\zp)^*$ of 
       $D(\mathcal{H}^{i+1}, \zp)$ and induces a natural map of $\Z_p$-modules
       $$r_{\eet}: {\rm Sp}_i(\zp)^*{\to} H^i_{\eet}(X,\Z_p(i)).$$
     \end{enumerate}
       \end{proposition}           

\begin{proof} Uniqueness in (1) is clear since the $\Z_p$-submodule of $D(\mathcal{H}^{i+1}, \zp)$ spanned by the Dirac distributions is dense. 
 
   Existence claim in (1) requires more work. 
        Let $\{U_n\}_{n\geq 1}$ be the standard admissible affinoid  covering of $X$. Let $\Pi(n)$ be the fundamental group of $U_n$. Denote by $\rg(\Pi(n),\Z_p(i))$ the complex of nonhomogenous continuous cochains representing the continuous group cohomology of $\Pi(n)$. By the $K(\pi,1)$-Theorem of Scholze \cite[Th. 1.2]{Sch1} this complex also represents $\rg_{\eet}(U_n,\Z_p(i))$. Since the action of $\Pi(n)$ on $\Z_p(1)$ is trivial the local  \'etale Chern class map factors as 
        \begin{equation*}
        c_{1,n}^{\eet}: \so(U_n)^*\to \Hom(\Pi(n),\Z_p(1))\to \rg(\Pi(n),\Z_p(1))[1].
        \end{equation*}
         The global \'etale Chern class is represented  by the composition
         $$
          c_{1,n}^{\eet}: \so(X)^*\to \invlim_n\so(U_n)^*\to \holim_n\Hom(\Pi(n),\Z_p(1))\to \holim_n\rg(\Pi(n),\Z_p(1))[1]\stackrel{\sim}{\leftarrow}  
          \rg_{\eet}(X,\Z_p(1))[1].
         $$
         The \'etale regulator $r_{\eet}:\so(X)^{*,\otimes i}\to    \rg_{\eet}(X,\Z_p(i))[i]$ is then  represented by the cup product:  $r_{\eet}:=c^{\eet}_1\cup\cdots \cup c^{\eet}_{1}$.
         
           The composition
           \begin{equation}
           \label{pozno1}
           \Psi_i: \sh^{i+1}\to \so(X)^{*,\otimes i}\lomapr{r_{\eet}}\rg_{\eet}(X,\Z_p(i))[i]
           \end{equation}
           represents the map $\psi_{\eet}$. We claim that it is continuous. Indeed, it suffices to show that so are the induced maps 
           $\Psi_{i,n}: \sh^{i+1}\to    \rg(\Pi(n),\Z_p(i))[i]$, for $n\geq 1$. Or, by continuty of the cup product that so are the maps $\Psi_{1,n}$. Or, simplifying further, that so are the maps
           \begin{equation}
           \label{pozno2}
               \Psi_{1,n}: \sh^{2}\to \so(X)^{*}\lomapr{r_{\eet}}\Hom(\Pi(n),\Z_p(1)).
           \end{equation}
           
        To show this, write 
$\mathcal{H}=\varprojlim_{m} \mathcal{H}_m$, where $\mathcal{H}_m$ is the set of $m$-equivalence classes of $K$-rational hyperplanes\footnote{Recall that two hyperplanes $H_1,H_2$ are called $m$-equivalent (i.e., $[H_1]=[H_2]\in \sh_m$) if they have unimodular equations $\ell_1, \ell_2$ such that $\ell_1= \ell_2$ modulo $\varpi^m$, where $\varpi$ is a uniformizer of $K$.} and
set  $M_n:=\Hom(\Pi(n),\Z_p(1))$. It suffices to show that, for each $k\geq 1$, there is an $m$ such that the map $$\Psi_{1,n,k}: \mathcal{H}^{2}\xrightarrow{\Psi_1} M_n\to M_n/p^kM_n$$ factors through the projection $ \mathcal{H}^{2}\to \mathcal{H}_m^{2}$.
   Taking into account the construction of $\Psi_{1,n}$, it suffices to show that, for $m$ large enough, if two hyperplanes 
   $H_0,H_1$ are $m$-equivalent, then $r_{\eet}(\ell_{H_1}/\ell_{H_0})\in p^k M_n$. But this is clear, since 
   in this case $\ell_{H_1}/\ell_{H_0}$ has a $p^{r_n m}$'th root in $\O(U_n)^{*}$, for some  constant $r_n>0$ depending only on $U_n$, and since $r_{\eet}$ is a homomorphism, we have $r_{\eet}(\ell_{H_1}/\ell_{H_0})\in p^{r_n m}M_n$.

     Since $\Hom(\Pi(n),\Z_p(1))$ is a Banach space and the map $\Psi_{1,n}$ from (\ref{pozno1}) is continuous on $\sh^{i+1}$, it defines, by integration against distibutions, 
      a continuous map
      $$
   r_{\eet,n}: D(\sh^{i+1},\Z_p)\to \rg_{\eet}(U_n,\Z_p(i))[i]
      $$
      such that $r_{\eet,n}(\delta_{(H_0,\ldots, H_i)})=\psi_{\eet,n}(H_0,\ldots, H_i)$, for all $H_0,\ldots, H_i\in \sh$, where $\psi_{\eet,n}$ is the analog of $\psi_{\eet}$ for $U_n$. The construction being compatible with the change of $n$ we get the existence of the map in (1) by setting 
      $r_{\eet}:=\invlim_nr_{\eet,n}$ and passing to cohomology. 
      
      For (2) we need to check  the factorization of the regulator $r_{\eet}$ from (1) through the quotient by the degenerate distributions. That is, we need to show that,
for any $\mu\in D(\mathcal{H}^{i+1}, \zp)_{\rm deg}$, we have $r_{\eet}(\mu)=0$. For that, by the construction of 
$r_{\eet}(\mu)$,  it suffices to check that:
\begin{enumerate}
\item  for all $H_0,...,H_{i+1}\in \mathcal{H}$, we have 
\begin{equation}
\label{cafe}\psi_{\eet}(H_1,...,H_{i+1})-\psi_{\eet}(H_0, H_2,..., H_{i+1})+...+(-1)^{i+1}\psi_{\eet}(H_0,...,H_{i})=0
\end{equation}
\item 
       and, if the
       $\ell_{H_j}$, $0\leq j\leq i$,  are linearly dependent, then     
       $\psi_{\eet}(H_0,...,H_i)=0$. 
       \end{enumerate}
       To see (1) note that (a)  if one permutes $H_j$ with $H_{j+1}$ in the expression in (\ref{cafe}) the latter will change its sign, (b) the coefficient of $r_{\eet}(\delta_{H_1,\ldots, H_{i+1}})$ in the development of
       the expression in (\ref{cafe}) is zero. 

(2) follows from the fact that the \'etale regulator satisfies the Steinberg relations. More precisely, if $x_j=\ell_j/\ell_0$, $0\leq j\leq i$, where $\ell_0,\ldots, \ell_i$ are linear equations of $K$-rational hyperplanes, it suffices to show that the symbol
$$\{x_1,\ldots, x_i, 1 +a_1x_1+\cdots +a_ix_i\}=0, \quad a_j\in K,$$  in the Milnor $K$-theory group $K^M_{i+1}(\so(X)^*).$  Note that the symbol $\{x_1,\ldots, x_i,1\}=0$. We will reduce to this case by the following algorithm. Step 1: up to reordering we may assume that $y_1:=(1+a_1x_1)\neq 0$ (otherwise we are done). Then, using the Steinberg relations $\{z,1-z\}=0$ and the fact that $\{x,a\}=0,$ for $ a\in K^*$,  we compute
$$
\{x_1,x_2, \ldots, x_i, 1 +a_1x_1+\cdots +a_ix_i\}=\{x_1,x_2/y_1, \ldots, x_i/y_1, 1 +a_2x_2/y_1+\cdots +a_ix_i/y_1\}.
$$
Note that this makes sense since $x_j/y_1\in\so(X)^*$. In fact, $x_j/y_1=\ell_j/(\ell_0+a_1\ell_1)$ is again a quotient of two linear equations of $K$-hyperplanes.
Step 2: reorder the terms in the last symbol to make $x_2/y_1$ appear first and repeat.

       \end{proof}

          \subsubsection{Integrating  $A_{\rm inf}$-symbols.} 
          Let $i\geq 1$. We pass now to the $\A_{\rm inf}$-regulator map
    \begin{align*}
       r_{\rm inf}:  A_{\rm inf}\wh{\otimes}_{\Z_p}{\rm Sp}_i(\zp)^*\to H^i_{\eet}(\calx, \ao\{i\})
     \end{align*}
      that is  compatible with the classical $A_{\rm inf}$-regulator as well as with the \'etale regulator
      $$
        r_{\eet}:  {\rm Sp}_i(\zp)^*\to H^i_{\eet}(X, \Z_p(i))
      $$
      defined above. To start, 
we define the regulators
    $$r_{\rm inf}:  D(\mathcal{H}^{i+1}, \zp)\to H^i_{\eet}(\calx, \ao\{i\}),\quad  r_{\rm inf}:  {\rm Sp}_i(\zp)^*\to H^i_{\eet}(\calx, \ao\{i\})$$
   by setting  $r_{\rm inf}:=\gamma r_{\eet}, $ 
    where $\gamma: H^i_{\eet}(X,\Z_p(i))\to H^i_{\eet}(\calx, \ao\{i\})$ is the canonical map and the \'etale regulator 
    \begin{equation}
    \label{sama}
    r_{\eet}:  D(\mathcal{H}^{i+1}, \zp)\to H^i_{\eet}(X, \Z_p(i))
    \end{equation}
     is the map defined above.

\begin{corollary}
\label{bursztyn2}Let $i\geq 1$. The above regulators extend uniquely to compatible 
continuous $A_{\rm inf}$-linear maps
 $$r_{\rm inf}:  A_{\rm inf}\wh{\otimes}_{\Z_p}D(\mathcal{H}^{i+1}, \zp)\to H^i_{\eet}(\calx, \ao\{i\}),\quad r_{\rm inf}: A_{\rm inf}\wh{\otimes}_{\Z_p}{\rm Sp}_i(\zp)^*{\to} H^i_{\eet}(\calx, \ao\{i\})
 $$
 that are compatible with the \'etale regulators.      
  \end{corollary}           
\begin{proof} Uniqueness is clear. To show the existence, let $\{U_n\}_{n\in\N}$ be the standard admissible affinoid covering of $X$. For $n\in\N$, set
$$r_{{\rm inf},n}:  D(\mathcal{H}^{i+1}, \zp)\to \rg_{\eet}(\su_n, A\Omega_{\su_n}\{i\})[i],\quad r_{\rm inf,n}:=\gamma r_{\eet,n},
$$
where $\su_n$ is the standard semistable formal model of $U_n$ and the map 
$$
r_{\eet,n}: D(\mathcal{H}^{i+1}, \zp)\to \rg_{\eet}(U_n, \Z_p(i))[i]
$$
was constructed above. The map $r_{{\rm inf},n}$ factors as
$$
r_{{\rm inf},n}:  D(\mathcal{H}^{i+1}, \zp)\to M_n\to \rg_{\eet}(\su_n, A\Omega_{\su_n}\{i\})[i],
$$
where $M_n:=\Hom(\Pi(n),\Z_p(i))$ for the fundamental group $\Pi(n)$ of $U_n$.

Since $r_{\eet}=\holim_nr_{\eet,n}$, we have the factorization
\begin{equation*}
r_{\eet}:  D(\mathcal{H}^{i+1}, \zp)\to \varprojlim_nM_n\verylomapr{\holim_nr_{\eet,n}} \holim_n\rg_{\eet}(U_n,\Z_p(i))[i]\simeq\rg_{\eet}(X,\Z_p(i))[i].
\end{equation*}
This induces the following factorization
\begin{align*}
r_{\rm inf}=\holim_nr_{{\rm inf},n}:  \quad & A_{\rm inf}\wh{\otimes}_{\Z_p}D(\mathcal{H}^{i+1}, \zp)\to A_{\rm inf}\wh{\otimes}_{\Z_p}\varprojlim_nM_n\stackrel{\sim}{\to}\varprojlim_nA_{\rm inf}\wh{\otimes}_{\Z_p}M_n
\\
 & \quad \to \holim_n\rg_{\eet}(\su_n, \ao\{i\})[i]\stackrel{\sim}{\leftarrow}\rg_{\eet}(\calx, \ao\{i\})[i].
\end{align*}
The existence of the first map is clear and so is the following isomorphism. The third map exists because both $A_{\rm inf}\wh{\otimes}_{\Z_p}M_n$ and $\rg_{\eet}(\su_n, \ao\{i\})$ are derived $(p,\mu)$-adically  complete.
 This proves the existence of the first regulator in the corollary. The existence of the second 
 follows immediately from the fact that the map (\ref{sama}) factors through ${\rm Sp}_i(\zp)^*$ once we know that the sequence
$$
 0\to A_{\rm inf}\wh{\otimes}_{\Z_p}D(\mathcal{H}^{i+1}, \zp)_{\rm deg} \to A_{\rm inf}\wh{\otimes}_{\Z_p}D(\mathcal{H}^{i+1}, \zp)\to A_{\rm inf}\wh{\otimes}_{\Z_p}{\rm Sp}_i(\zp)^*\to 0
$$
is strict exact. This sequence is obtained from the strict exact sequence (\ref{strict1}) by tensoring with $A_{\rm inf}$. Hence the only question is the strict surjection on the right, which follows from the fact that the sequence
(\ref{strict1}) is actually split (since all modules are duals of free modules). 
\end{proof}
     
     \subsection{The $A_{\rm inf}$-cohomology of  Drinfeld symmetric spaces}
We are now ready to prove Theorem \ref{main}.  Let $i\geq 0$. We will show that 
       the map $r_{\inf}$  induces a $\phi^{-1}$-equivariant topological  isomorphism of  $A_{\rm inf}$-modules
       $$r_{\rm inf}: A_{\rm inf}\widehat{\otimes}_{\zp} {\rm Sp}_i(\zp)^*\stackrel{\sim}{\to} H^i_{\eet}(\calx, \ao\{i\}).$$

Compatibility with the operator $\phi^{-1}$ follows from the fact that  $r_{\rm inf}$ is $A_{\rm inf}$-linear and it  is induced from $r_{\eet}$ hence maps 
 $D(\mathcal{H}^{i+1}, \zp)$ to $H^i_{\eet}(\calx, \ao\{i\})^{\phi^{-1}=1}$.
 For the rest of the claim, first, we show that the induced map $$
 \overline{r}_{\rm inf}: (A_{\rm inf}\widehat{\otimes}_{\zp} {\rm Sp}_i(\zp)^*)/\tilde{\xi}\to  H^i_{\eet}(\calx, \ao\{i\})/\tilde{\xi}
 $$
 is a topological isomorphism.
 But this map fits into the following commutative diagram (of $A_{\rm inf}$-linear continuous maps, where $A_{\rm inf}$ acts on $\so_C$ via $\tilde{\theta}$)
  $$
 \xymatrix{
 (A_{\rm inf}\widehat{\otimes}_{\zp} {\rm Sp}_i(\zp)^*)/\tilde{\xi}\ar[r]^-{ \overline{r}_{\rm inf}} \ar[d]^{\tilde{\theta}}_{\wr} & H^i_{\eet}(\calx, \ao\{i\})/\tilde{\xi}\ar@{^(->}[d]^{\alpha}\ar@/^50pt/[dd]^{\overline{\iota}_{\hk}} \\
   \so_C\widehat{\otimes}_{\zp} {\rm Sp}_i(\zp)^*\ar[rd]^{r_{\htt}}_{\sim}& H^i_{\eet}(\calx, \ao\{i\}/\tilde{\xi})\ar[d]_{\wr}\\
 & H^0_{\eet}(\calx, \Omega^i_{\calx})
 }
 $$
 The map $\alpha$ is the  change-of-coefficients  map;  it is clearly injective.  The right vertical map is an isomorphism because we have the local-global spectral sequence
 $$
 E_2^{s,t}=H^s_{\eet}(\calx, H^t(\ao\{i\}/\tilde{\xi}))\Rightarrow H^{s+t}_{\eet}(\calx, \ao\{i\}/\tilde{\xi})
 $$
 and, by Theorem \ref{dRspec} and Theorem \ref{CK}, the isomorphisms $H^t(\ao\{i\}/\tilde{\xi})\simeq H^t(\wt{\Omega}_{\calx}\{i\})\simeq \Omega^t_{\calx}\{i-t\}$. Hence, by Theorem \ref{GK}, 
 $$
  E_2^{s,t}=H^s_{\eet}(\calx,  \Omega^t\{i-t\})=0,\quad s\geq 1.
 $$
 The above  diagram  commutes by Proposition \ref{compatibility2}. The slanted arrow is a topological  isomorphism by Corollary  \ref{integraldR}. It follows that the map $\alpha$ is surjective, hence it is an isomorphism and so is, by the above diagram, the  map $ \overline{r}_{\rm inf}$.  The latter is also a topological isomorphism because so is the map $\tilde{\theta}$ and the map $\overline{\iota}_{\hk}$ is a continuous isomorphism. 
 
  Next, we will show that $\overline{r}_{\rm inf}$ being a topological  isomorphism implies that so is the original map ${r}_{\rm inf}$.   
  Let $T$ be the homotopy fiber of ${r}_{\rm inf}$. We claim that 
  the complex 
  \begin{equation}
  \label{derived-Nakayama}
  T\otimes^{\rm L}_{A_{\rm inf}}A_{\rm inf}/\tilde{\xi}\simeq 0.
  \end{equation} Indeed, since $\overline{r}_{\rm inf}$ is an isomorphism, it suffices to show that 
 the domain and the target of ${r}_{\rm inf}$ are  $\tilde{\xi}$-torsion free. This is clear for the domain. For the  target, 
note that 
 the exact triangle 
 $$\ao\{i\}\lomapr{\tilde{\xi}} \ao\{i\}\to \ao\{i\}/\tilde{\xi}$$
yields an exact sequence
 $$0\to H^i_{\eet}(\calx, \ao\{i\})/\tilde{\xi}\lomapr{\alpha} H^i_{\eet}(\calx, \ao/\tilde{\xi})\to H^{i+1}_{\eet}(\calx, \ao\{i\})[\tilde{\xi}]\to 0.$$
 By the above, the first map is an isomorphism, hence
  $H^{i+1}_{\eet}(\calx, \ao\{i\})[\tilde{\xi}]=0$.   Since $i\geq 0$ was arbitrary, we deduce that, for all $j\geq 1$ and all $i$, 
 $H^j_{\eet}(\calx, \ao\{i\})$ has no $\tilde{\xi}$-torsion, and this is clearly true for $j=0$ as well. 
 
  Since $T$ is derived $\tilde{\xi}$-complete (because so are the domain and the target of $r_{\inf}$, the latter using the derived $\tilde{\xi}$-completeness of $\ao$ and the preservation of this property by derived pushforward and passage to cohomology), by the derived Nakayama Lemma (see Section \ref{derived-complete}) we have $T\simeq 0$ as well. 
   This finishes the proof that $r_{\rm inf}$ is an isomorphism. 
  
   Since the domain and the target of $r_{\rm inf}$ are $\tilde{\xi}$-torsion-free and the reduction $\overline{r}_{\rm inf}$ is a topological isomorphism so is $r_{\rm inf}$.
  This finishes the proof.

\section{Integral $p$-adic \'etale cohomology of Drinfeld symmetric spaces}We are now ready to compute the \'etale cohomology. 
  Let  $X_K:={\mathbb H}^d_K$ be the Drinfeld symmetric space of dimension $d$ over $K$ and let  $\calx_{\so_K}$   be its standard semistable formal model over $\so_K$. Let $X:=X_K\times_KC$.
   \begin{theorem}
   \label{final} Let $i\geq 0$. 
   \begin{enumerate}
   \item There is a  $G\times \sg_K$-equivariant  topological  isomorphism
   $$
   r_{\eet}: {\rm Sp}_i(\Z_p)^*\stackrel{\sim}{\to} H^i_{\eet}(X, {\Z}_p(i)).
   $$
 It is compatible with the rational isomorphism $r_{\eet}: {\rm Sp}_i(\Z_p)^*\otimes \Q_p\stackrel{\sim}{\to} H^i_{\eet}(X, {\Q}_p(i))$ from \cite{CDN3}.
 \item There is a  $G\times \sg_K$-equivariant  topological  isomorphism
   $$
   \overline{r}_{\eet}: {\rm Sp}_i({\mathbf F}_p)^*\stackrel{\sim}{\to} H^i_{\eet}(X, {\mathbf F}_p(i)).
   $$
   \end{enumerate}
   \end{theorem}
   \begin{proof} 
Set  $\calx:=\calx_{\so_C}$.
        For $i\geq 0$, using the natural  isomorphism $H^i_{\eet}(X, {\Z}_p(i))\stackrel{\sim}{\to}H^i_{\proet}(X, \wh{\Z}_p(i))$ \cite[proof of Cor. 3.46]{CDN3}, we pass to pro-\'etale cohomology. Now, by Corollary \ref{BMS22}, we have a natural  short exact sequence 
          \begin{equation}
          \label{evening11}
          0\to {H^{i-1}_{\eet}(\calx, \ao\{i\})}/(1-\phi^{-1})\to H^i_{\proet}(X, \wh{\Z}_p(i))\to H^i_{\eet}(\calx, \ao\{i\})^{\phi^{-1}=1}\to 0.
          \end{equation}
      By Theorem \ref{main}, we have a topological isomorphism $ H^i_{\eet}(\calx, \ao\{i\})\simeq A_{\rm inf}\wh{\otimes}_{\Z_p}{\rm Sp}_i(\Z_p)^*$ and this isomorphism is compatible with the action of $\phi^{-1}$. We get  topological isomorphisms
             \begin{align*}
              H^i_{\eet}(\calx, \ao\{i\})^{\phi^{-1} =1}& \simeq (A_{\rm inf}\wh{\otimes}_{\Z_p}{\rm Sp}_i(\Z_p)^*)^{\phi^{-1}=1}\simeq A_{\rm inf}^{\phi^{-1}=1}\wh{\otimes}_{\Z_p}{\rm Sp}_i(\Z_p)^*\simeq {\rm Sp}_i(\Z_p)^*,\\
              {H^{i-1}_{\eet}(\calx, \ao\{i\})}/(1-\phi^{-1})  & \simeq (A_{\rm inf}\wh{\otimes}_{\Z_p}{\rm Sp}_{i-1}(\Z_p)^*)/(1-\phi^{-1})\simeq (A_{\rm inf}/({1-\phi^{-1}}))\wh{\otimes}_{\Z_p}{\rm Sp}_i(\Z_p)^*\simeq 0.
             \end{align*}
             Hence, by the exact sequence (\ref{evening11}), we get a natural continuous  isomorphism $r_{\proet}: {\rm Sp}_i(\Z_p)^*\stackrel{\sim}{\to}H^i_{\proet}(X, \wh{\Z}_p(i))$.      Since its composition with the natural map 
             $H^i_{\proet}(X, \wh{\Z}_p(i))\stackrel{\sim}{\to}H^i_{\eet}(\calx, \ao\{i\})^{\phi^{-1}=1}$ is a topological isomorphism so is the map $r_{\proet}$ itself, as wanted in claim (1). 
                    
   The last sentence of claim (1) of the theorem is clear. 

  For claim (2), we define the regulator $\overline{r}_{\eet}$ in an analogous way to its integral version $r_{\eet}$ (with which it is compatible by construction). Since 
  ${\rm Sp}_i({\mathbf F}_p)^*\simeq {\rm Sp}_i({\mathbf Z}_p)^*\otimes{\mathbf F}_p$ and $H^i_{\eet}(X, {\mathbf F}_p(i))\simeq H^i_{\eet}(X, {\mathbf Z}_p(i))\otimes{\mathbf F}_p$ (the latter isomorphism by claim (1)), we have $\overline{r}_{\eet}\simeq r_{\eet}\otimes\id_{{\mathbf F}_p}$. Hence, by claim (1),  $\overline{r}_{\eet}$ is an isomorphism, as wanted. 
\end{proof}


\begin{thebibliography}{Bam0}
 \bibitem{BerN} V.~Berkovich, {\em  Smooth $p$-adic analytic spaces are locally contractible}. Invent. Math. 137 (1999), 1--84.
  \bibitem{BerL} V.~Berkovich, {\em  Complex analytic vanishing cycles for formal schemes}, preprint.
\bibitem{BK} S.~Bloch, K.~Kato, {\em $p$-adic \'etale cohomology},  Inst. Hautes \'Etudes Sci. Publ. Math. 63 (1986), 107--152.
 \bibitem{Bha} B.~Bhatt, {\em Specializing varieties and their cohomology from characteristic $0$ to characteristic $p$},  Algebraic geometry: Salt Lake City 2015, 43--88, Proc. Sympos. Pure Math., 97.2, Amer. Math. Soc., Providence, RI, 2018.
\bibitem{BMS1} B.~Bhatt, M.~Morrow, P.~Scholze, {\em Integral $p$-adic Hodge Theory},   Inst. Hautes \'Etudes Sci. Publ. Math.	 128  (2018),  219--397.
 \bibitem{BMS2} B.~Bhatt, M.~Morrow, P.~Scholze, {\em Topological Hochschild homology and integral $p$-adic Hodge theory},  	Inst. Hautes \'Etudes Sci. Publ. Math.	 129  (2019),  199--310.
 \bibitem{CK} K.~\v{C}esnavi\v{c}ius, T.~Koshikawa, {\em The $A_{\rm inf}$-cohomology in the semistable case}, arXiv:1710.06145v2 [math.NT].
 \bibitem{CDN1} P.~Colmez, G.~Dospinescu, W.~Nizio\l,
{\em Cohomologie p-adique de la tour de Drinfeld: le cas de la dimension 1}, arXiv:arXiv:1704.08928 [math.NT].
 \bibitem{CDN3} P.~Colmez, G.~Dospinescu, W.~Nizio\l, {\em Cohomology of $p$-adic Stein spaces}, arXiv:1801.06686 [math.NT].
\bibitem{Drinfeld}
{V.~Drinfeld},
{\em Elliptic modules}, Math. Sb.~{94} (1974), 594-627.
 \bibitem{Fo82} J.-M. Fontaine, {\em Formes diff\'erentielles et modules de Tate des vari\'et\'es ab\'eliennes sur les corps locaux}.  Invent. Math. 65 (1981/82), no. 3, 379--409. 
\bibitem{FDP}  J.~Fresnel, M.~van der Put, {\em  Rigid analytic geometry and its applications}. Progr. Math. 218, Birk\"auser, 2004.
 \bibitem{GK1} E.~ Grosse-Kl\"onne, {\em Integral structures in the $p$-adic holomorphic discrete series}. Represent. Theory 9 (2005), 354--384. 
  \bibitem{GKF} E.~Grosse-Kl\"onne, {\em Frobenius and monodromy operators in rigid analysis, and Drinfeld's symmetric space}. J. Algebraic Geom. 14 (2005), 391--437. 
 \bibitem{GK2} E.~Grosse-Kl\"onne, {\em  Sheaves of bounded $p$-adic logarithmic differential forms}. Ann. Sci. \'Ecole Norm. Sup. 40 (2007), 351--386. 
 \bibitem{GKir} E.~Grosse-Kl\"onne, {\em  On special representations of $p$-adic reductive groups}. Duke Math. J. 163 (2014), 2179--2216.
 \bibitem{IS} A.~Iovita,  M.~Spiess, {\em Logarithmic differential forms on $p$-adic symmetric spaces}. Duke Math. J. 110 (2001), 253--278. 
\bibitem{Morr} M.~Morrow, {\em $p$-adic vanishing cycles as Frobenius-fixed points},  	arXiv:1802.03317 [math.AG].
\bibitem{SS} P.~Schneider, U.~Stuhler, {\em The cohomology of $p$-adic symmetric spaces}. Invent. Math. 105 (1991), 47--122. 
\bibitem{Sch1} P.~Scholze, {\em $p$-adic Hodge theory for rigid-analytic varieties.} Forum Math. Pi 1 (2013), e1, 77 pp.
 \bibitem{Surv} P.~Scholze, {\em Perfectoid spaces: a survey. Current developments in mathematics 2012}, 193--227, Int. Press, Somerville, MA, 2013.
 \bibitem{SP} The Stacks Project Authors, Stacks Project, http://stacks.math.columbia.edu.
 \end{thebibliography}
\end{document}